\documentclass[11pt]{amsart} 

\usepackage{amsmath}
\usepackage{amsthm}
\usepackage{amscd}
\usepackage{amssymb}
\usepackage{graphicx}
\usepackage{latexsym} 
\usepackage[all]{xy}
\usepackage{enumerate}
\usepackage{color}
\usepackage[dvipsnames]{xcolor}
\usepackage{fullpage}

\newcommand{\cal}{\mathcal}

\def\epsilon{\varepsilon}
\def\phi{\varphi}

\def\hat{\widehat}
\newcommand{\C}{\mathbb C}

\newcommand{\Aut}{\mbox{Aut}}

\newcommand{\R}{\mathbb R}
\newcommand{\Z}{\mathbb Z}

\newcommand{\N}{\mathbb N}

\def\strutdepth{\dp\strutbox}
\def \ss{\strut\vadjust{\kern-\strutdepth \sss}}
\def \sss{\vtop to \strutdepth{
\baselineskip\strutdepth\vss\llap{$\diamondsuit\;\;$}\null}}

\def\strutdepth{\dp\strutbox}
\def \sst{\strut\vadjust{\kern-\strutdepth \ssss}}
\def \ssss{\vtop to \strutdepth{
\baselineskip\strutdepth\vss\llap{$\spadesuit\;\;$}\null}}

\def\strutdepth{\dp\strutbox}
\def \ssh{\strut\vadjust{\kern-\strutdepth \sssh}}
\def \sssh{\vtop to \strutdepth{
\baselineskip\strutdepth\vss\llap{$\heartsuit\;\;$}\null}}

\def\qed{\hfill\rlap{$\sqcup$}$\sqcap$\par}
\def\bar{\overline}

\def\strutdepth{\dp\strutbox}
\def \ss{\strut\vadjust{\kern-\strutdepth \sss}}
\def \sss{\vtop to \strutdepth{
\baselineskip\strutdepth\vss\llap{$\diamondsuit\;\;$}\null}}

\def\strutdepth{\dp\strutbox}
\def \sst{\strut\vadjust{\kern-\strutdepth \ssss}}
\def \ssss{\vtop to \strutdepth{
\baselineskip\strutdepth\vss\llap{$\spadesuit\;\;$}\null}}

\def\qed{\hfill\rlap{$\sqcup$}$\sqcap$\par}

\newtheorem{thm}{Theorem}[section]
\newtheorem{cor}[thm]{Corollary}
\newtheorem{lem}[thm]{Lemma}

\newtheorem{prop}[thm]{Proposition}

\theoremstyle{definition}

\newtheorem{defn}[thm]{Definition}

\newtheorem{convention-defn}[thm]{Convention-Definition}
\newtheorem{defn-rem}[thm]{Definition-Remark}
\newtheorem{rem}[thm]{Remark}
\newtheorem{remark}[thm]{Remark}

\numberwithin{equation}{section}

\begin{document} 

\author[Martin Lustig]{Martin Lustig}
\address{\tt Aix Marseille Universit\'e, CNRS, Centrale Marseille, I2M, UMR 7373, 
13453  Marseille, France}
 \email{\tt Martin.Lustig@univ-amu.fr} 

\author[Yoav Moriah]{Yoav Moriah} 
 \address{\tt Department of Mathematics at the Technion, IIT,
 Haifa
 Israel 32000
}
 \email{\tt ymoriah@technion.ac.il} 
 



\title[Nielsen equivalence]
{Nielsen equivalence in Fuchsian groups}
 
\begin{abstract} 
In this paper we give a complete classification of minimal generating systems in a very general class 
of Fuchsian groups $G$. This class includes for example any $G$ which has at least seven non-conjugate 
cyclic subgroups of order $\gamma_i \geq 3$. In particular, the well known 
problematic cases where $G$ has characteristic exponents $\gamma_i = 2$ are  not excluded.

We classify generating systems up to Nielsen equivalence; this notion is strongly related to Heegaard splittings 
of $3$-manifolds. The results of this paper provide in particular the tools for a rather general extension of 
previous work of the authors and others, on the isotopy classification of such splittings in Seifert fibered spaces. 
\end{abstract}

\date{\today}
 
\maketitle 

\section{Introduction}
\bigskip
A Fuchsian group is a group that acts properly discontinuously and cocompactly by isometries on the hyperbolic 
plane $\mathbb{H}^2$. If $G$ preserves the orientation of $\mathbb{H}^2$, then it is a discrete subgroup of 
$PSl_2(\R),$ and it has a presentation
\begin{equation}
\label{Fuchs-1}
G = \langle s_1, \dots, s_\ell, a_1, b_1, \dots, a_g, b_g  \mid s_1^{\gamma_1},\dots, 
s_\ell^{\gamma_\ell} , s_1 s_2 \dots s_\ell \overset{g}{\underset{ j = 1}{ \Pi}}  [a_j, b_j] \rangle\, ,
\end{equation}
with $\gamma_i \geq 2$ for all $i = 1, \ldots, \ell$. The isomorphism type of such a Fuchsian group is 
determined by the family of {\em characteristic exponents} $\gamma_i$ and by the {\em genus $g \geq 0$ of $G$}.
A presentation for Fuchsian groups with glide reflections is given below in (\ref{non-orientable-Fuchs}).

\smallskip

Fuchsian groups play a central role in both, hyperbolic geometry and in low-dimensional topology
(see \cite{Bea:GeomDesc}, \cite{ Kra:Book}, \cite{Th}). Most prominently, all prime 3-dimensional 
manifolds were shown by Thurston to  have a natural geometric structure, with eight possible 
geometries. For six of these geometries, the corresponding $3$-manifolds are Seifert fibered spaces.
Orientable  Seifert fibered spaces with an orientable base space have fundamental groups that are 
central extensions of Fuchsian groups $G$ as in (\ref{Fuchs-1}) above 
(see \cite{Scott:geom}). 

\smallskip

It is well known that the presence of characteristic exponents $\gamma_i = 2$ in (\ref{Fuchs-1})  
often creates serious problems, in the sense that otherwise well working arguments fail in this case. 
For example, if $g=0$ and all but one exponent satisfy $\gamma_i = 2$, then 
the rank of $G$ (= the minimal number of generators) can unexpectedly drop by 1, leading 
to an intriguing phenomenon in the corresponding Seifert fibered space (see \cite{BoiZie:HeegGen}, 
\cite{PeRoZie}). The most important achievement of this paper is that in our main result, Theorem
 ~\ref{main} stated below, exponents $\gamma_i = 2$ are not excluded, and even the 
difficult case where the number of such $\gamma_i$ is odd, is dealt with.

In order to simplify the presentation and concentrate on the crucial issues, we treat in the main body of this 
paper only the orientable case, and we assume $g = 0$. The general case, including the possibility of 
orientation reversing isometries of $\mathbb H^2$, can subsequently be deduced directly from this special case, 
see Section \ref{sec:Generalization}.

\smallskip

From now on let $G$ be a group with presentation
\begin{equation}\label{Fuchs}
G = \langle s_1, \dots, s_\ell \mid s_1^{\gamma_1},\dots, s_\ell^{\gamma_\ell} , s_1 s_2 
\dots s_\ell \rangle\, ,
\end{equation}
with $\gamma_i \geq 2$ for $i = 1, \ldots, \ell$, and with $\ell \geq 3$. Let $m$ denote the number 
of standard 
generators $s_i$ with exponent $\gamma_i \geq 3$, and let $n = \ell - m$ denote the number of 
those $s_i$ with exponent $\gamma_i = 2$. Note that  the indexing of the generators is immaterial 
for the isomorphism type of $G$, as any permutation  of the $s_i$ can be obtained through iteratively
replacing some $s_i$ by
\begin{equation}
\label{gernerator-permutation}
s'_i = s_{i-1}^{-1} s_i s_{i-1}\, \,, \,  \text{
which yields} \,\, \,\,s_{i-1} s' _i = s_i s_{i -1} \, .
\end{equation}

The group $G$ can be generated by $\ell - 1$ elements, and G. Rosenberger has shown in \cite{Ros89}, \cite{Ros92} 
(compare also \cite{Ros74}, \cite{Roes:Aut}, \cite{Zie}, \cite{Zie1}) that for $\ell \geq 4$ and $m \geq 3$
any minimal generating system  can be transformed by a sequence of elementary  {\em Nielsen operations} 
(see Definition \ref{Nielsen-equivalence-M}  below) into a generating system of the following type: 

\begin{defn} \label{def:scheme}
A   family ${\mathcal{U}}$ of elements in $G$ is called {\em a standard generating system} of $G$ if 
\begin{equation}\label{U}
{\mathcal{U}} = (s_1^{u_1}, \ldots, s_{j-1}^{u_{j-1}}, s_{j+1}^{u_{j+1}}, \ldots, s_\ell^{u_\ell})\, ,
\end{equation}
with $\gcd(u_i, \gamma_i) = 1$ for all $i \in \{1, \ldots, j-1, j+1, \ldots, \ell\}$.
\end{defn}

The main goal of this paper is to present a complete proof of the following:

\begin{thm} \label{main}
Let $G$ be a  group as in (\ref{Fuchs}), and let $m$ be the number of $\gamma_i \geq 3$. 
If the number $n = \ell - m$ of exponents $\gamma_i = 2$ is even, we assume that $m \geq 5$. 
If $n$ is odd, assume $m \geq 7$. 

Assume that $\,\mathcal{U}$ as in (\ref{U}) and 
$${\mathcal{V}} = (s_1^{v_1}, \ldots, s_{k-1}^{v_{k-1}}, s_{k+1}^{v_{k+1}}, \ldots, s_\ell^{v_\ell})$$
are two given standard generating systems of $G$. Define $u_j = v_k = 1$. 
Then $\mathcal{U}$ and $\mathcal{V}$ are Nielsen equivalent if and only if 
$$u_i = \pm v_i \quad \text{mod} \quad \gamma_i \qquad \text{for all} \qquad i = 1, \ldots, \ell\, .$$
\end{thm}

\begin{remark} 
\label{rem:Generalization}
(1) From the proof presented in this paper it follows that the conclusion of Theorem \ref{main} 
is valid under weaker assumptions than those stated for $m$ and $n$:  It suffices that $G$ is 
``non-exceptional'' as in Definition \ref{exceptional} below (which is a bit too cumbersome 
to be stated here).

\smallskip
\noindent 
(2) A direct generalization of Theorem \ref{main} to orientable Fuchsian groups with genus $g \geq 1$ 
and to non-orientable Fuchsian groups is given below in Corollary \ref{main+}. The methods presented in the 
subsequent sections are actually stable enough to include an application to more general 1-relator quotients 
of free products of cyclic groups, see Theorem \ref{W-extension}.

\smallskip
\noindent 
(3) An extension of our methods beyond what is presented in this paper, in order to make the 
set of exceptional groups $G$ even smaller, is possible, but the technical effort becomes increasingly 
bigger if $m$ and the seize of the exponents $\gamma_i$ 
becomes smaller. 

\smallskip
\noindent (4) It is known that for very special choices of $m$ and $n$, the conclusion of Theorem \ref{main} 
fails (for instance take $G$ as in \cite{PeRoZie}). Other candidates for such a failure have been proposed for 
example by G. Rosenberger.  The precise determination of all Fuchsian groups with standard generating systems 
that admit more Nielsen equivalences than expected from our result above  remains  unresolved and  seems to 
be a very difficult problem.

\end{remark}

The study of Nielsen equivalence for generating systems of groups has a long history: It has been a central 
theme in combinatorial group theory since the 1950's,  for example in the context of non-tame automorphisms 
of groups (see e.g. \cite{MoSh}).  Even with Gromov's paradigm change towards geometric group theory 
in the 1990's, its relevance has not decreased (see e.g. \cite{DowdallTaylor}, \cite{HeusnerWeidemann}, 
\cite{KapovichWeidemann}, \cite{KapWeid04}, \cite{Louder}, \cite{MyNa}, \cite{Weidemann:GenSys} or, more 
classically,  \cite{PeRoZie}, \cite{Roes:Aut}, \cite{Zie1}). In fact, it has also spread into other branches of 
mathematics \cite{Bridson}, \cite{Islambouli},  \cite{Lub}, \cite{Souto}, as well as to computer science 
\cite{B-LG}, \cite{Wigd}.

Among the various natural reasons to investigate Nielsen equivalence of generating systems, 
one of the most important ones comes from the study of compact $3$-dimensional manifolds $M^3$: 
Every Heegaard splitting of $M^3$ determines two generating systems of $G = \pi_1(M^3)$ up to Nielsen 
equivalence, and  an isotopy of the splitting preserves the Nielsen equivalence classes. Indeed, the latter 
are the most telling and also most useful invariants  of such splittings, and in the majority of cases non-isotopic 
Heegaard splittings are distinguished by these invariants.

The authors of this paper have in previous work (see \cite{Lus:NieEquivSim}, \cite{LuMo:NEinFuch} and 
\cite{LuMoGenSysTor}) developed the fundamentals of the method used here, and set up a K-theoretic invariant
 $\cal N(G)$ to distinguish minimal generating systems in arbitrary groups (see Remark \ref{N-torsion-question} 
below). This has led, by work of the second author with J. Schultens (see \cite{MorSch:VertHorHS}
and \cite{Sch:HSwithBounday}), to a classification of minimal genus Heegaard splittings in a large class of 
Seifert fibered spaces, excluding, however, those where the underlying Fuchsian groups 
have characteristic exponents $\gamma_i = 2$.

In the present paper, instead of employing the powerful $\cal N(G)$ machinery, we only 
need ``Jacobian matrices'' defined via Fox derivatives over $\Z G$ (see Section 2), as well 
as a special evaluation technique, via ``cyclic-faithful'' representations of $G$ in $Sl_2(\C)$ 
(see Sections 3 and 4). Our final calculations take place in $2 \times 2$-matrices over a group 
ring of a cyclic group with coefficients in $\C$ (see Section 5), and various cases have to be 
considered that stretch over a number of pages (Sections 6 and 7).

This paper is in many ways a continuation of our previous work \cite{Lus:NieEquivSim}, \cite{LuMo:NEinFuch} and  
\cite{LuMoGenSysTor}. For the convenience of the reader, however, we present here a self-contained exposition, 
and we also make a special effort to organize the (non-trivial) computational parts of the paper into ``compartments'' 
where they can be checked independently from the presentation of our main arguments.

We'd also like to point the reader's attention to recent work \cite{Dutra1} of Edison Dutra, as well as to the upcoming 
papers \cite{Dutra2} by Dutra and \cite{Dutra-W} by Dutra-Weidmann on related questions.

\medskip

\noindent {\em Acknowledgements:}
The authors would like to thank the referee for several very valuable comments which helped to improve 
and correct an earlier version of this paper.  Furthermore, we would like to thank Gerhard Rosenberger for 
straightening out some of our more classical references. We would also like to thank Wendy Sandler as 
well as David Kohel for advice about some intricacies of the English language.


\section{Preliminaries} \label{prelims}
\bigskip
In this section we  briefly review the notions of Fox derivatives and Nielsen equivalence.

\subsection{Fox derivatives}

The notion of Fox derivatives was developed by R. Fox in \cite{Fox:DiffCal}. For a 
modern exposition see \cite{BuZie:Knots}. 

\begin{defn}\label{Fox-derivative-M}
Let $X = (X_1, \ldots , X_n)$ be a basis of a free group $F_n$. Then the {\em $i$-th Fox derivative 
with respect to $X$} is a $\mathbb{Z}$-linear  map 
 $$\partial / \partial X_i : \Z F_n \to \Z F_n\, ,\, \,W \, \mapsto \,\,\partial W/\partial X_i \, ,$$
which satisfies (where $\delta_{i,j}$ denotes the Kronecker-delta)
\begin{enumerate}
\item   $ \partial X_j / \partial X_i = \delta_{i,j}$ for any $j \in \{1, \ldots, n\}$, and
\vskip5pt
\item $\partial (U\cdot V) / \partial X_i  = \partial U / \partial X_i  + U\cdot \partial V / \partial X_i$ 
for any $U, V \in F_n\,$.
\end{enumerate}

The maps $\partial/ \partial X_i$ are characterized by these two properties and the assumed $\Z$-linearity.
Note that the notation $\partial/\partial X_i$ is slightly misleading, as the map $\partial/\partial X_i$ depends 
not only on $X_i$, but also on the choice of the other $X_j$ from the given basis $X$.
\end{defn}

Fox derivatives have many natural uses in algebra and topology, and they turn out to be fairly easy to handle. 
For example, for any $i \in \{1, \ldots, n\}$ one can immediately derive from (1) and (2) above the following facts: 
For the neutral element $1 \in F_n$  one has $\partial 1 / \partial X_i = 0$, and  for any $W \in F(X)$ 
the formula 
\begin{equation}
\label{inverse-formula}
\partial W^{-1} / \partial X_i = -W^{-1} \cdot \partial W/\partial X_i \, . 
\end{equation}
Furthermore, for any $V, W \in F_n$ the equality
\begin{equation}
\label{inverses-M}
\partial (W V W^{-1}) / \partial X_i = W\, \partial V/\partial X_i + (1 - W V W^{-1}) \, \partial W / \partial X_i \,.
\end{equation}
is satisfied.

\smallskip

Let $Y = (Y_1, \ldots , Y_m)$ be a second basis of $F_n$. Then for any element $W \in F_n$ we 
have the {\em chain rule}: 
\begin{equation}\label{ChainRule}
\partial W/ \partial X_i = \sum_{k = 1}^{n}(\partial W /\partial Y_k   \cdot \partial Y_k / \partial X_i)
\end{equation}
Hence the $n$-tuple $(\partial W /\partial X_i)_{i= 1, \ldots, n}$ is the matrix product of the 
$(1 \times n)$-matrix $(\partial W /\partial Y_k)_{k = 1, \ldots, n}$ with the {\em Jacobian matrix} 
\begin{equation}\label{Jacobian} 
\partial Y/ \partial X = (\partial Y_k/ \partial X_i )_{k, i = 1, \ldots, n}
\end{equation}
over the group ring $\mathbb{Z}F_n$. This  matrix is invertible over $\mathbb{Z}F_n$: 
From property (1) in Definition \ref{Fox-derivative-M} and a direct application of the chain rule  
one obtains 
$\partial X/ \partial Y \cdot \partial Y/ \partial X = \partial X / \partial X = I_n$
(where $I_n$ denotes the $n\times n$  identity matrix).

\medskip

Let $\cal U = (x_1, \ldots, x_n)$ a generating system for a group $G$. Then for the 
free group $F(X)$ over a family $X = (X_1, \ldots, X_n)$ of formal symbols $X_i$ there 
is a canonical surjection
\begin{equation}
\label{canonical-surjection}
p_\cal U: F(X) \twoheadrightarrow G , \, X_i \mapsto x_i \,.
\end{equation}
Any element $w \in G$ can be written as a ``word'' in $\cal U$, i.e. 
\begin{equation} \label{word-M}
w = \omega_1 \omega_2  \ldots \omega_r \quad \text{with} \quad \omega_i 
\in \{x_1^{\pm 1}, \ldots, x_r^{\pm 1}\}. 
\end{equation}
Lifting each $x_i$ to $X_i$ determines an element $W \in F(X)$ so that 
$$p_\cal U(W) = w\, .$$
The $n$ Fox derivatives $\partial W/ \partial X_i$, when mapped into $\Z G$ via the ring 
homomorphisms $\Z F(X) \to \Z G$ induced by $p_\cal U$ (and hence also denoted by 
$p_\cal U$), give rise to an $n$-tuple
\begin{equation}
\label{Fox-images-M}
\partial w/\partial \cal U = (p_\cal U(\partial W/\partial X_1), \ldots, p_\cal U(\partial W/\partial X_n))
\end{equation}
Any other word $w^* = \omega ^*_1 \ldots \omega ^*_{r^*}$ as in (\ref{word-M}), which describes the 
same element
$$w = w^* \,\,\, \text{in} \,\,\,  G\, ,$$
gives rise to a second lift $W^* \in F(X)$, which differs from $W$ by an element 
$R = W^* W^{-1} \in \ker p_\cal U$. If furthermore $\ker p_\cal U$ is normally generated by the 
elements of a set $\cal R = \{R_1, \ldots, R_m\}$, we have
$$W^* = (V_1 R_{j_1}^{\epsilon_1} V_1^{-1} \ldots V_q R_{j_q}^{\epsilon_q} V_q^{-1}) W$$
for suitable $V_h \in F(X)$,  $R_{j_h} \in \cal R$ and $\epsilon_h = \pm 1$. Hence we derive, 
from property (2) of Definition \ref{Fox-derivative-M} and from formula (\ref{inverses-M}), that
\begin{equation}\label{sum-expression-M}
p_\cal U(\partial W^*/\partial X_i) = p_\cal U\left(\sum_{j = 1}^{m} \big{(}\sum_{\{h \, \mid \, j_h = j\} }  
\epsilon_h V_h \big{)}  \cdot \partial R_j /\partial X_i \right) + p_\cal U\left(\partial W/ \partial X_i\right)
\end{equation}
for any $i \in \{1, \ldots, n\}$. As a consequence, we obtain from (\ref{Fox-images-M}) that 
\begin{equation}
\label{1-line}
\partial w^* / \partial \cal U = \partial w / \partial \cal U + L \, ,
\end{equation}
where each entry of the $n$-tuple $L$ is of the same type as the first term in the sum on the right hand 
side of equality (\ref{sum-expression-M}). We formalize this observation as follows:

\begin{defn-rem} \label{correction-matrix-M}
For any group $G$ and any generating system $\,\cal U = (x_1, \ldots, x_n)$ of $G$ consider 
the canonical surjection 
$$p_\cal U: F(X_1, \ldots, X_n) \twoheadrightarrow G, \, \, X_i \mapsto x_i \, .$$

\smallskip
\noindent (1) A matrix $B \in M_n(\Z G)$ is called a {\em correction matrix} if every coefficient 
of $B$ is contained in the left ideal $I_\cal U^{\ell}$ of $\Z G$ which is generated by the Fox 
derivative images  $p_\cal U(\partial R / \partial X_i)$, for any $R \in \ker p_\cal U$ and $X_i$ 
from $X = (X_1, \ldots, X_n)$.

\smallskip
\noindent (2) If $\cal R = \{R_1, \ldots, R_m\}$ is a set of normal generators of $\ker p_\cal U$, 
then $I_\cal U^{\ell}$ is generated as a left ideal in $\Z G$  by all $p_\cal U(\partial R_j / \partial X_i)$ 
with $R_j \in \cal R$ and $X_i$ from $X$. This  is the content of (\ref{sum-expression-M}), for the 
case $W = 1 \in F(X)$.
\end{defn-rem}

 Now (\ref{1-line}) implies directly:

\begin{prop}\label{Jacobian-M}
Let $\,\cal U = (x_1, \ldots, x_n)$ be a generating system of a group $G$. 
Consider a second generating system $(y_1, \ldots, y_n)$ of $G$, and assume that 
each $y_k$ is expressed as a word $w_k$ in $\cal U$.  
Then the collection $\cal W = (w_1, \ldots, w_n)$ determines a ``Jacobian matrix'' 
$\partial \cal W / \partial \cal U = (a_{k,i})_{k, i} \in M_n (\Z G)$,  
where $a_{k, i} = p_\cal U(\partial w_k/ \partial X_i)$ 
(so that the $k$-th  line of $\partial \cal W / \partial \cal U$ is defined as in 
(\ref{Fox-images-M}), with $w_k$ replacing $w$).

Let $\cal W^*$ be a second  collection of such words $w^*_k$ for each $y_k$. Then there is a correction 
matrix $B$ such that the two Jacobian  matrices associated to $\cal W$ and $\cal W^*$ satisfy:
$$\partial \cal W^* / \partial \cal U = \partial \cal W / \partial \cal U + B$$
\qed
\end{prop}

Note that, contrary to $\partial Y/\partial X$ in (\ref{Jacobian}), the more general Jacobian 
matrix $\partial \cal W / \partial \cal U$ in the above proposition is in general {\em not} invertible over $\Z G$. 
An example is given by the generator $y = x^2$ of the cyclic group $G = \langle x \mid x^5 \rangle$.
For $\cal U = (x)$ and $\cal W = (x^2)$ the matrix $\partial \cal W / \partial \cal U$  is the $(1 \times 1)$-matrix 
with coefficient  $1+x$, which is not invertible in $\Z G$.


\subsection{Nielsen equivalence} 

${}^{}$

\begin{defn} \label{Nielsen-equivalence-M}
 Let $G$ be a group, let $n \in \N$, and let $ \cal U = (x_1, \ldots, x_n)$ be an $n$-tuple of 
elements from $G$. Then an {\em elementary Nielsen operation} on $\cal U$ is given by 
one of the following:
\begin{enumerate}
\item a permutation of the $x_i$,
\item replace $x_i$ by $x_i x_j$ or by $x_j x_i$, for $j \neq i$, while all other members of $\cal U$ 
stay unchanged, or
\item replace $x_i$ by $x_i^{-1}$, while all other members of $\cal U$ stay unchanged.
\end{enumerate}
A finite sequence of elementary Nielsen operations is sometimes called a {\em Nielsen operation}, 
and two families $\cal U$ and $\cal U'$ are {\em Nielsen equivalent} if they can be derived from 
each other by Nielsen operations.
\end{defn}

\begin{rem}\label{quotient-Nielsen}
Let $f: G \to H$ be a group homomorphism, let $\cal U$ and $\cal U'$ be families of elements in $G$, 
and denote by $f(\cal U)$ and $f(\cal U')$ the families of their $f$-images in $H$. If $\cal U$ and $\cal U'$ 
are Nielsen equivalent, then so are $f(\cal U)$ and $f(\cal U')$. This is an immediate consequence of 
Definition \ref{Nielsen-equivalence-M}.
\end{rem}

Nielsen operations have been introduced by J. Nielsen in the 20's of the last century, as analogues 
of elementary row operations on integer matrices. He could then show that bases for a free group 
$F_n$ have the  property described in the following theorem, in analogy to what is well known for 
bases of free abelian groups $\Z^n$:

\begin{thm}
\label{Nielsen-M}
Let $X = (X_1, \ldots, X_n)$ and $Y = (Y_1, \ldots, Y_n)$ be two bases of a free group $F_n$. 
Then there exists a finite sequence of elementary Nielsen operations that transform $X$ into $Y$.

Conversely, if $X$ is a basis of $F_n$ and $Y$ derives from $X$ by a finite sequence of Nielsen 
operations, then $Y$ is also a basis of $F_n$.
\qed
\end{thm}

Contrary to rings like $\Z$ or $\R[X]$, for non-commutative groups $G$ 
the units (= multiplicatively invertible elements) in $\Z G$ may in general be 
quite complicated.  However, within the multiplicative group of units in $\Z G$ 
there is always the subgroup of {\em trivial units}, given by
$$T_G = \{ \pm g \mid g \in G \} \, .$$

\begin{defn}
\label{gen-elem-matrix}
For any group $G$ we say that a square matrix $M$ with entries in $\Z G$ is called 
a {\em generalized elementary matrix} over $\Z G$. if $M$ satisfies one of the following:
\begin{enumerate}
\item
$M$ is a permutation matrix,
\item
$M$ differs from the identity matrix only in a single off-diagonal coefficient, or
\item
$M$ is a diagonal matrix with trivial units on the diagonal.
\end{enumerate}
We will refer to these matrices as {\em elementary $\Z G$-matrices}
\end{defn}

\begin{prop}
\label{decomposition-M}
Let $X = (X_1, \ldots, X_n)$ and $Y = (Y_1, \ldots, Y_n)$ be two bases of a free group $F_n$.
Then the Jacobian matrix
$$\partial Y / \partial X = (\partial Y_j / \partial X_i)_{j, i}$$
is a product of elementary $\Z F_n$-matrices.
\end{prop}

\begin{proof}
If $Y$ is derived from $X$ by a single elementary Nielsen operation, the claimed statement follows 
from a direct computation based on (1) and (2) in Definition \ref{Fox-derivative-M}.
The full claim is thus an immediate consequence of Theorem \ref{Nielsen-M} 
and the fact that Fox derivatives satisfy the chain rule, see (\ref{ChainRule}).
\end{proof}

 Combining Proposition \ref{decomposition-M} with Proposition \ref{Jacobian-M} gives immediately  
the main criterion used in this paper to detect Nielsen inequivalent generating systems in 
an arbitrary finitely generated group $G$:

\begin{prop}
\label{criterion-M}
Let $\cal U = (x_1, \ldots, x_n)$ and $\cal V = (y_1, \ldots, y_n)$ be two Nielsen equivalent generating 
systems of a group $G$. For any family of expressions 
$$y_1 = w_1\, ,\, \ldots \,, \,\,y_n = w_n$$
of the $y_k$ as words $w_k$ in the generators $x_i$, and their canonical lifts $W_k \in F(X)$ under 
the surjection $p_\cal U: F(X) \to G, \, X_i \mapsto x_i$, consider the Jacobian matrix 
$\partial \cal V / \partial \cal U = (p_\cal U (\partial W_k / \partial X_i))_{k, i}$. 
Then there is a correction matrix $B \in {\mathbb M}_n(\Z G)$ as in Definition \ref{correction-matrix-M} 
such that the sum
$$\partial \cal V / \partial \cal U + B$$ 
is a product of elementary $\Z G$-matrices.
\qed
\end{prop}

This proposition is particularly useful in combination with a suitable homomorphism of the group ring 
$\Z G$ into a matrix ring. For instance, let $\zeta_5$ be a 5-th root of unity, and for the cyclic group 
$G = \langle x \mid x^5 \rangle$ consider the homomorphism $\eta: \Z G \to \C$ given by $x \mapsto \zeta_5$. 
Then one has $\partial x^5 / \partial x = 1 + x + \ldots + x^4$, which gives 
$\eta(\partial x^5 / \partial x) = 1 + \zeta_5 + \ldots + \zeta_5^4 = 0 \in \C$. 

As a consequence, any correction matrix $B$ is mapped by $\eta$ to the zero-matrix, so that the image 
$\eta(\partial \cal V/ \partial \cal U)$ of the Jacobian matrix $\partial \cal V/ \partial \cal U$ as in 
Proposition \ref{criterion-M} depends only on the generators $y_k$ in $G$ and not on their lifts $W_k$ 
to $F(X)$. As a direct application of Proposition \ref{criterion-M} we obtain that the generating system 
$\cal V$ which consists of the generator $y = x^2$ is not Nielsen equivalent to the generating system 
$\cal U = (x)$, since the element $\eta(\partial \cal V/ \partial \cal U) = 1 + \zeta_5$ is not a product 
of $\eta$-images of elementary $\Z G$-matrices (of seize $(1 \times 1)$), as the latter are all 
equal a power of $\pm \zeta_5$.

Of course, we knew ahead of these considerations that $x^2$ and $x$ are not Nielsen equivalent in the cyclic 
group $G$, but with very little effort the above reasoning is extended to the free product $G * \Z$ or even $G * F_n$, 
where Nielsen equivalence is a less obvious issue.

\medskip

For the purpose of this paper homomorphisms of $\Z G$ into commutative rings $A$ are not sufficient; 
instead, we need to consider  homomorphisms $\eta$ of $\Z G$ into matrix rings ${\mathbb M}_m(A)$. 
By a slight abuse of notation we denote by $\eta$ also the induced map from ${\mathbb M}_n(\Z G)$ to 
${\mathbb M}_{m \cdot n}(A)$, to obtain:

\begin{cor}\label{evaluation-M}
Let $G, \, \cal U$ and $\cal V$ be as in Proposition \ref{criterion-M}. Let $A$ be a commutative 
ring, and for some integer $m \geq 1$ let $\eta: \Z G \to {\mathbb M}_m(A)$ denote a ring homomorphism
which satisfies $\det \eta(g) = \det \eta(-g) = 1$ for all $g \in G$.

Then there exists a ``correction term'' $b \in I_\cal U^A$ such that 
$$\det \eta(\partial \cal V / \partial \cal U) + b = 1\, ,$$
where $I_\cal U^A$ is the ideal in $A$ generated by all coefficients of the $\eta$-image 
of any Fox derivative matrix $p_\cal U(\partial R_j / \partial X_i)$, for any set of normal generators $R_j$ of 
$\ker (p_\cal U: F(X) \to G)$.
\end{cor}

\begin{proof}
It suffices to apply $\eta$ to the sum $\partial \cal V / \partial \cal U + B$ from Proposition \ref{criterion-M} 
and to consider the determinant of the resulting $(m \cdot n \times m \cdot n)$-matrix over $A$. Since $B$ 
is a correction matrix in ${\mathbb M}_n(\Z G)$, its coefficients are all contained in the left ideal $I_\cal U^\ell$ 
of $\Z G$ that is generated by the Fox derivative images $p_\cal U(\partial R_j / \partial X_i)$. Hence 
the determinant $\det \eta(\partial \cal V / \partial \cal U + B)$ is a sum of $\det \eta(\partial \cal V / \partial \cal U)$ 
with terms that are all products which contain at least one of the coefficients of the $\eta$-image of some of the 
$p_\cal U(\partial R_j / \partial X_i)$ as factor, so that they all belong to $I_\cal U^A$.

The claim then is a direct consequence of the assumption that any trivial unit $\pm g$ of $\Z G$ has as $\eta$-image 
the value $1 \in A$, since from Proposition \ref{criterion-M} we know that $\partial \cal V / \partial \cal U + B$ is the 
product of elementary $\Z G$-matrices, and (according to the description in Definition \ref{gen-elem-matrix}) any 
such matrix has as $\eta$-image a matrix with the same determinant as the $\eta$-image of some trivial unit.
\end{proof}



\section{Cyclic-faithful representations}\label{sec:FuchsianGroups}

Let $G$ be a group as in (\ref{Fuchs}), i.e. 
$$G = \langle s_1, \dots, s_\ell \mid s_1^{\gamma_1},\dots, s_\ell^{\gamma_\ell} , s_1 s_2 
\dots s_\ell \rangle\, ,$$
with $\gamma_i \geq 2$ for $i = 1, \ldots, \ell$, and with $\ell \geq 3$. If in addition $G$ satisfies

\begin{equation}\label{Fuchs-criterion}
\sum_{i = 1}^\ell \frac{1}{\gamma_i} < \ell - 2
\end{equation}

then it is a Fuchsian group. Thus there is a faithful representation

\begin{equation}\label{representation}
\rho_0: G \rightarrowtail PSl_2(\C) \, .
\end{equation}

It is well known (see \cite{Culler} and \cite{Kra:Book}, pp. 181--193) that $\rho_0$ lifts to a faithful 
representation

\begin{equation}
\label{faithful-rep}
\rho: G \rightarrowtail Sl_2(\C)
\end{equation}

if and only if all exponents  $\gamma_i$ in (\ref{Fuchs}) are odd. Furthermore,  every standard 
generator  $s_i$ of $G$ is mapped by $\rho$, up to conjugation in $Sl_2(\C)$, to a matrix of type

\begin{equation}
\label{M-zeta-}
M(\zeta_i) = \left[
\begin{array}{cc}\zeta_i & 0\\
0 & \zeta_i^{-1}
\end{array} \right] \, ,
\end{equation}

where $\zeta_i \in \C$ is a primitive $\gamma_i$-th root of unity.  Matrices such as $M(\zeta_i)$
will be called {\it primitive $\gamma_i $-matrices}.

In this paper we will use representations in $Sl_2(\C)$ which are slightly more general 
in that they need not be faithful on all of $G$:

\begin{defn}
\label{cyclic-faith}
For any $G$ as in (\ref{Fuchs}) a representation $\rho: G \to Sl_2(\C)$ will be called 
{\em cyclic-faithful}  if $\rho$ maps every standard generator $s_i$ to a conjugate of a 
primitive $\gamma_i$-matrix. 
\end{defn}

\begin{rem}
\label{element-lift}
Regarding Definition \ref{cyclic-faith} we note:
\smallskip

\noindent (1) The terminology ``cyclic-faithful'' is justified, since the defining property 
of $\rho$ is equivalent to requiring that $\rho$ is faithful when restricted to the cyclic 
subgroup generated by any of the standard generators.

\smallskip
\noindent (2) Let $\rho_0: G \to PSl_2(\C)$ be faithful, and consider for every generator $s_i$ 
both lifts of $\rho_0(s_i)$ in $Sl_2(\C)$. If $\gamma_i$ is odd, then precisely  one of these two 
lifts will have order $\gamma_i$, while the other has order $2 \gamma_i$. If $\gamma_i$ is even, 
then both lifts will have order $2 \gamma_i$. 

\smallskip
\noindent (3) For the special case $\gamma_i = 2$ 
we recall that 
one has $\zeta_i = -1$ and thus $M(\zeta_i) = M(-1) = - I_2\,$, where as before $I_2$ denotes the 
$2 \times 2$ identity matrix. Indeed, $M(-1) = - I_2$ is the only matrix in $Sl_2(\C)$ which has order 
two.
\end{rem}

In order to find cyclic-faithful representations of $G$ it is  useful to introduce a certain  canonical 
quotient of $G$. Since there are  two similar such quotients,  we will introduce them here 
together, so that the reader will avoid confusion later on.

\begin{defn} \label{2-quotient}
Let $G$ be as in (\ref{Fuchs}).

\smallskip
\noindent  (1) Set $\gamma'_i = \frac{\gamma_i}{2}$ 
if $\gamma_i$ is even and $\gamma'_i = \gamma_i$ if $\gamma_i$ is odd.
Define the {\em full 2-quotient}:
$$G^* = G / \langle \langle \{s_i^{\gamma'_i} \mid i = 1, \ldots, \ell \}\rangle\rangle$$

\smallskip
\noindent 
(2) The {\em canonical 4-quotient} of $G$ is given by
$$G^\# = G / \langle \langle \{s_i^{\hat\gamma_i} \mid i = 1, \ldots, \ell \}\rangle\rangle \, ,$$
where we set $\hat \gamma_i = \frac{\gamma_i}{2}$ if $\gamma_i$ is even,  but not divisible by 4
nor equal to 2, and otherwise we set $\hat \gamma_i = \gamma_i$.
\end{defn}

\begin{rem}\label{quotients}
We note that the full 2-quotient $G^*$ is in general generated by fewer elements than $G$, since any 
$s_j$ which in $G$ has order $\gamma_j = 2$ will be trivial in $G^*$. 

The canonical $4$-quotient $G^\#$, on the other hand, will in almost all cases\footnote{\, The only exceptions 
occur if $G^\#$ is one of the groups studied in \cite{BoiZie:HeegGen} and \cite{PeRoZie}, which have already 
been mentioned in the introduction.}
be of the same rank as $G$.  It has the useful property that any generator $s_i$ is mapped in $G^\#$ to an 
element of order which is either odd, equal to 2, or divisible by 4. Furthermore $G^\#$ is ``stable'' in the sense 
that  $(G^\#)^\# = G^\#$.
\end{rem}

\begin{rem}\label{lifts}
In order to find a cyclic-faithful representation $\rho$ of a Fuchsian group $G$ as in (\ref{Fuchs}), 
our strategy is to first pass to the quotient $G^*$, then use a faithful representation $\rho_0$ of this 
quotient group in $PSl_2(\C)$, and finally define the images $\rho(s_i)$ as suitable lifts of $\rho_0(s_i)$.
According to Remark \ref{element-lift} (2), if properly chosen, these lifts $\rho(s_i)$ are all conjugates of 
primitive $\gamma_i$-matrices, where $\gamma_i$ is the original exponent of $s_i$ in $G$. There are, 
however, three obstructions to overcome, when attempting this procedure:
\begin{enumerate}
\item 
The quotient group $G^*$ may not  be Fuchsian. Hence, in order to ensure the existence of 
$\rho_0$ as above, one has to verify the  inequality 

\begin{equation}
\label{eq3.5}
\sum_{\{i \,\mid\, \gamma_i \geq 3\}} \frac{1}{\gamma'_i} < m -2 \, .
\end{equation}
for $\gamma'_i$ as defined in Definition \ref{2-quotient} (1), and $m$ equal to the number of standard 
generators $s_i$ with exponent $\gamma_i \geq 3$.

\item For the generators $s_j$ of order $\gamma_j =2$ the above ``lifting trick'' doesn't work:
As noted already in Remark \ref{element-lift} (3), the only element of $Sl_2(\C)$ of order 2 is 
the matrix $-I_2$, which is also equal to $M(\zeta_i)$ as in (\ref{M-zeta-}). Hence any 
cyclic-faithful representation of $G$ must satisfy the equality
\begin{equation}
\label{eq3.6}
\rho(s_j) = -I_2 \, ,
\end{equation}
for any $s_j$ with exponent $\gamma_j = 2$.

\item Even if (\ref{eq3.5}) and (\ref{eq3.6}) above are satisfied, it may still be that the {\em product relation} 
$s_1 s_2 \ldots s_\ell = 1$ does not hold for the chosen $\rho$-images of the $s_i$. If, however, one has

\begin{equation}\label{product-relation}
\rho(s_1) \rho(s_2) \ldots \rho(s_\ell) = I_2 \, ,
\end{equation}
then the above definition of the $\rho(s_i)$ defines a representation $\rho: G \to Sl_2(\C)$ 
which is cyclic-faithful.
\end{enumerate}
\end{rem}

In the following section  several methods which ensure the existence of such cyclic-faithful 
representations $\rho$ are presented. It turns out that satisfying equality (\ref{product-relation}), 
in the case where $n$ is odd, is surprisingly tricky.



\section{Exceptional Fuchsian groups}\label{sec:RepOfG}

We start this section by listing conditions on groups $G$ as in (\ref{Fuchs})
 which ensure the existence of a cyclic-faithful representation of $G$ into 
$Sl_2(\C)$. We then define ``exceptional'' Fuchsian groups, and show that 
any non-exceptional $G$ satisfies one of these conditions.

\medskip

\begin{prop} \label{faithful}
Let $G$ be as in (\ref{Fuchs}).  Let $n$ denote the number of exponents 
$\gamma_j = 2$, and let $m$ denote the number of exponents $\gamma_i \geq 3$.
 Assume that one of the following conditions is satisfied:
\begin{enumerate}
\item There is at least one exponent $\gamma_i$ which is divisible by 4\,. Further\-more the inequality 
\begin{equation}\label{fake-sum}
\sum_{\{ i \, \mid \, \gamma_i \geq 3\}} 
\frac{1}{\gamma'_i} < m - 2
\end{equation}
is satisfied, where, as in Definition \ref{2-quotient}, we set $\gamma'_i = \frac{\gamma_i}{2}$
 if $\gamma_i$ is even and $\gamma'_i = \gamma_i$ if $\gamma_i$ is odd.
\vskip4pt
\item Every exponent $\gamma_i \neq 2$ is odd, and the number $n \geq 0$ of exponents 
$\gamma_j = 2$ is even. Assume furthermore that 
\begin{enumerate}
\vskip4pt
\item[(i)] $m \geq 4$, or
\vskip4pt
\item[(ii)] $m = 3$, and there is at least one $\gamma_i \geq 5$.
\end{enumerate}
\smallskip
\item Every exponent $\gamma_i \neq 2$ is odd, the number $n$ is odd and $m \geq 6$.
\vskip4pt
\item Every exponent $\gamma_i \neq 2$ is odd,  the number $n$ is odd and one of the 
following is true:
\vskip4pt

\begin{enumerate}
\item[(i)]  $m = 5$, and there is some $\gamma_i \geq 7$.
\vskip4pt

\item[(ii)] $m = 5$, and  there are at least two $\gamma_i, \gamma_{i'} \geq 5$.
\vskip4pt

\item[(iii)]
$m = 4$, and there are  at least two $\gamma_i, \gamma_{i'} \geq 7$.
\vskip4pt

\item[(iv)] $m = 4$, and all four exponents satisfy $\gamma_1, \gamma_2,\gamma_3, \gamma_4 \geq 5$.
\end{enumerate}
\end{enumerate}
\vskip4pt
Then there is a cyclic-faithful representation 
$$\rho: G \to Sl_2(\C) \, .$$

\end{prop}

\begin{proof}
Consider the four cases in order:

\smallskip
\noindent \underline{Case (1):} Proceed exactly as in Remark \ref{lifts}:  Assumption
 (\ref{fake-sum})  ensures (see (\ref{Fuchs-criterion})) the existence of a faithful representation  
 $\rho_0: G^* \to P Sl_2(\C)$ of the full 2-quotient 
  $G^* = G / \langle \langle \{s_i^{\gamma'_i} \mid i = 1, \ldots, \ell \}\rangle\rangle$.
By assumption one of the $s_i$ has order $\gamma_i$ in $G$, where $\gamma_i$ is 
divisible by 4; thus both lifts of $\rho_0(s_i)$ to $Sl_2(\C)$ have order $\gamma_i$
(see Remark \ref{element-lift} (2)). Hence the right choice of $\rho(s_i)$  ensures that 
(\ref{product-relation}) is satisfied, and thus  $\rho$ is a cyclic-faithful representation of $G$.

\smallskip
\noindent \underline{Case (2):}
Proceed again as in Remark \ref{lifts}. The assumption that all $\gamma_i \geq 3$ are odd 
implies $\gamma'_i = \gamma_i$ for all such $\gamma_i$. Hence the assumptions (i) or (ii) ensure 
that inequality (\ref{fake-sum}) is satisfied and thus $\rho_0$ exists. As all $\gamma_i \geq 3$ are 
assumed to be odd, $\rho_0$ lifts to a faithful representation of $G^*$ in $Sl_2(\C)$, see (\ref{faithful-rep}).

In this case the number $n$ of generators $s_j$ of order $\gamma_j = 2$ is  even. We proceed 
as in step (2) of Remark \ref{element-lift} and extend the above lift of $\rho_0$ by the equalities 
(\ref{eq3.5}) to all generators of $G$. Since $n$ is even,  the product relation (\ref{product-relation}) 
is satisfied, so that we obtain the desired  cyclic-faithful representation $\rho$.

\smallskip
\noindent \underline{Case (3):} Proceed first as in Case (2).  The assumption $m \geq 6$ 
ensures  that  inequality (\ref{fake-sum}) holds and hence $\rho_0$ exists as before. 
However, since in this case $n$ is odd, the product relation (\ref{product-relation}) fails by a 
factor of $-1$. 

In order to deal with this problem we introduce the following ``trick'': Using the assumption 
that $m \geq 6$, we partition the generators $s_i \geq 3$ into two sets $s_1, \dots, s_r$ and 
$s_{r+1}, \dots, s_m$ in such a way that both, $r \geq 3$ and $m-r \geq 3$, hold. 
Set   $s_0 = (s_1s_2 \dots s_r)^{-1}$  and consider the group 
$$G_1 = \langle s_0, \ldots, s_r \mid s_1^{\gamma_1}, \ldots, s_r^{\gamma_r}, s_0^4,
 \,\,s_0 s_1s_2 \dots  s_r \rangle$$
Note that the assumptions $m \geq 3$ and $\gamma'_i = \gamma_i$ for all $\gamma_i \geq 3$ 
(assumed to be odd) ensure that $G_1$ satisfies the conditions given in Case (1). 
Thus $G_1$ admits a cyclic-faithful representation $\rho_1: G_1 \to Sl_2(\C)$. 

In particular, $\rho_1$ maps the product $s_0^{-1} = s_1s_2  \dots s_r$ to a conjugate of the primitive 
$4$-matrices $M(i)$ or $M(-i)$ (see equality (\ref{M-zeta-})), so that after  conjugating $\rho_1$ suitably
in $Sl_2(\C)$ we can assume $\rho_1(s_1s_2 \dots s_r) = M(i)$ or $\rho_1(s_1s_2 \dots s_r) = M(-i)$.

Apply the same method to the generators $s_{r+1}, \dots, s_m$ to obtain a group $G_2$ and a 
representation $\rho_2: G_2 \to Sl_2(\C)$ which maps the product $(s_{r+1}\dots s_m)$ to $M(i)$ 
or to $M(-i)$. 
Let $\bar \rho_2: G \to Sl_2(\C)$ be the representation obtained from $\rho_2$ through replacing, 
in the $2 \times 2$ image matrix of any element of $G_2\,$, each coefficient by its complex conjugate. 

It follows that combining  $\rho_1$ with either $\rho_2$, or $\rho_1$ with  $\bar\rho_2$, will map the product 
$s_{1}s_{2} \dots s_m$ to $M(i)^2 = -I_2$. Thus we can again proceed as in step (2) of Remark \ref{element-lift}
and extend this representation by the equalities (\ref{eq3.6}) to all generators of $G$: Since $n$ is odd, now  
the product relation (\ref{product-relation}) is satisfied, so that we obtain again a cyclic-faithful representation 
$\rho$ as claimed.

\smallskip
\noindent \underline{Case (4):} 
In order to apply the same trick as in the previous case,  extra arguments are needed to ensure that 
the cyclic-faithful representations $\rho_1$ and $\rho_2$ exist:

In the subcases (i) and (ii), in order to define $G_1$ and $G_2$ we partition (after reordering) the 
generators $s_1, \dots, s_5$ of order $\gamma_i \geq 3$ into two subsets $\{s_1, s_2, s_3\}$ and 
$\{s_4, s_5\}$. This partition is chosen so that for (i) either $s_4$ or $s_5$ has order $\geq 7$, and 
for (ii) both $s_4$ and $s_5$ have order $\geq 5$. It follows that, after adding a generator $s_0$ as 
in Case (3) above, both $G_1$ and $G_2$ satisfy the inequality (\ref{fake-sum}): 
Indeed, for $G_2$ the corresponding triple sum of exponents is smaller or equal to  
$\frac{1}{7} + \frac{1}{3} + \frac{1}{2} < 1$ 
(for (i)) or $\frac{1}{5} + \frac{1}{5} + \frac{1}{2} < 1$ (for (ii)).

In the subcases (iii) and (iv) the partition is $\{s_1, s_2\}\, \cup \, \{s_3, s_4\}$, where both 
sides are treated precisely as the subset $\{s_4, s_5\}$ in subcases (i) and (ii) above.

This ensures that both resulting  ``partial quotient groups'' $G_1$ and $G_2$ admit cyclic-faithful 
representations $\rho_1$ and $\rho_2$ as in Case (3) above. The rest of the  proof of Case (3) applies 
word-by-word.
${}^{}$
\end{proof}


\begin{defn} 
\label{exceptional}
Let $G$ be a group as in (\ref{Fuchs}), i.e. 
$$G = \langle s_1, \ldots, s_\ell \mid s_1^{\gamma_1}, \ldots, s_\ell^{\gamma_\ell}, 
s_1 s_2 \ldots s_\ell \rangle$$ with all $\gamma_i \geq 2$.  To simplify notation, we assume that the indexing 
of the generators has been adjusted (using equality (\ref{gernerator-permutation})) to achieve 
$\gamma_i \geq \gamma_j$ if $i \leq j$.

Let $m$ be the number of exponents $\gamma_i \geq 3$, and $n\geq 0$ be the number of exponents 
$\gamma_j = 2$. 
In this case we say $G$ {\em is of type} 
$$(\gamma_1, \ldots ,\gamma_m \mid n)\, .$$ 

Furthermore we use  below the following convention: for any odd integer $k \geq 0$ we write $k^*$ to 
include both cases, $k^* = k$ or $k^* = 2k$. For even $k$ the term $k^*$ is purposefully undefined, 
so that the use of $k^*$ implies in particular that $k$ is odd.

Then $G$ is called {\em exceptional} if 
one of the following conditions is satisfied. 

\begin{enumerate}
\item[(a)] The number of exponents $\gamma_i \geq 3$ satisfies $m \leq 3$.
\vskip4pt
\item[(b)] $n$ is even and $G$ is of type $(\gamma_1, 6, 5, 4 \mid n)$ or $(\gamma_1, 5, 4, 3 \mid n)$. 
\vskip4pt
\item[(c)] $n$ is even and $G$ is of type $(\gamma_1, \gamma_2, 4, 4 \mid n)$.
\vskip4pt
\item[(d)] $n$ is odd and $m = 4$.
\vskip4pt
\item[(e)] $n$ odd and $G$ is of type $(s^*, t^*, p^*, q^*, 3^* \mid n)$, with  $p, q \in \{3, 5\}$.
\vskip4pt
\item[(f)] $n$ odd and $G$ is of type $(p^*, q^*, 3^*, 3^*, 3^*, 3^* \mid n)$, with $q \in \{3, 5\}$.
\end{enumerate}
Otherwise $G$ is {\em non-exceptional}.
\end{defn}

\begin{rem} \label{exceptional-quotient}
The following two statements follow directly  from Definition \ref{exceptional}: 
\smallskip

\noindent (1) For any non-exceptional group $G$ the canonical 4-quotient $G^\#$ 
(see Definition \ref{2-quotient}) is also non-exceptional. Note that, in order 
for this statement to be true, in the Cases (e) and (f) one needs the above definition  of $k^*$.

\medskip
\noindent (2) Any group $G$ as in Theorem \ref{main} is non-exceptional.
\end{rem}

\begin{lem} \label{reps}
Let $G$ be a group as in (\ref{Fuchs}) which is non-exceptional, and assume that  
\begin{equation}\label{divisible}
\text{if} \,\,\gamma_i  \,\,\text{is even and} \,\, \gamma_i \neq 2 \,, \,\, \text{then} \,\, 
\gamma_i \,\, \text{is divisible by 4.}
\end{equation}
In other words, one has $G = G^\#$.
Assume also that for one of the standard generators, say $s_h$, 
 we have $\gamma_h \geq 5$. Then the following hold:
\begin{enumerate}
\item There is a cyclic-faithful representation $\rho: G \to Sl_2(\C)$.
\vskip4pt
\item  For any choice of $k\neq h$ the quotient group $G_0 = G/\langle\langle s_k\rangle\rangle$ admits 
a cyclic-faithful representation $\rho: G_0 \to Sl_2(\C)$.
\end{enumerate}
\end{lem}

\begin{proof} It follows immediately from Definition ~\ref{exceptional} that  $G$ being exceptional 
implies that $G_0$ is also exceptional. Hence it suffices to prove statement (2). This is done below by 
considering several cases and showing in each case that  $G_0$ satisfies one of the four conditions 
listed in Proposition \ref{faithful}. This shows the existence of the desired cyclic-faithful representation 
$\rho$.

\smallskip

The assumption that $G$ is non-exceptional implies that each of the conditions  (a) - (f) stipulated 
in Definition ~\ref{exceptional} is false. The negation of condition (a)  implies that for $G_0$ the 
number  $m_0$ of exponents $\gamma_i \geq 3$ satisfies: 
\begin{equation}
\label{m-bigger-3}
m_0 \geq 3
\end{equation}

There are two cases to be distinguished: 

\smallskip

\noindent (A) Assume that one of the $G_0$-exponents $\gamma_i$ is divisible by 4. 
This case splits further into  three subcases:
\smallskip

\noindent $(i)$ If all $\gamma_i \geq 3$ satisfy $\gamma_i \neq 4$,  then we have for $\gamma'_i$ 
(as in Definition ~\ref{2-quotient}) that $\frac{1}{\gamma'_i} \leq \frac{1}{3}$, if $i \neq h$.
From assumption (\ref{divisible}) we deduce furthermore that $\gamma_h \neq 6$ and thus 
$\frac{1}{\gamma'_h} \leq \frac{1}{5}$. Hence the inequality 
\begin{equation}\label{sum-4}
\sum_{\{ i \, \mid \, \gamma_i \geq 3, \,  i \neq k\}} \frac{1}{\gamma'_i} < m_0 - 2
\end{equation}
holds, and thus all assumptions from  Case (1) in Proposition \ref{faithful} are satisfied for $G_0$.

\smallskip

\noindent $(ii)$ 
Assume that for $G_0$ one has $m_0 \geq 4$. Then even if some 
$\gamma_i$ are equal to 4, we still have  $\frac{1}{\gamma'_i} \leq \frac{1}{2}$ for $i \neq h$, 
so that inequality (\ref{sum-4}) holds again, and we can apply the same conclusion as in subcase $(i)$.  

\smallskip

\noindent $(iii)$ In the remaining case there are precisely three exponents $\gamma_i \geq 3$ in 
$G_0$, among which we have $\gamma_h \geq 5$, and another exponent, say $\gamma_j$,
 which is equal to 4.  By the negation of conditions (c) and (d) in Definition \ref{exceptional} for $G$, 
 the third exponent $\gamma_i \geq 3$ must be different from 4. Furthermore, the negation of 
 conditions (b) and (d) rules out the possibility of  $\frac{1}{2} + \frac{1}{3} + \frac{1}{5}$ 
on the left hand side of the above inequality (\ref{sum-4}). For all other cases  inequality 
(\ref{sum-4}) is satisfied, so that  again we have the same  conclusion as in subcase $(i)$ above. 

\medskip

\noindent (B) In this case  all exponents $\gamma_i \geq 3$ are assumed to be odd,
 so that $\gamma'_i =\gamma_i$ holds for any $\gamma_i \geq 3$. 
 There are still two more subcases to consider:

\smallskip

\noindent $(iv)$ If  $n$  is even, then by (\ref{m-bigger-3}) the assumptions of conditions (i) or (ii) of Case (2) 
in Proposition \ref{faithful} are satisfied for $G_0$, due to our hypothesis that $\gamma_h \geq 5$.

\smallskip

\noindent $(v)$ If $n$ is odd, then the negation of  condition (d) in Definition \ref{exceptional} for 
$G$, together with condition (\ref{m-bigger-3}), ensure  that $G_0$ has  at least  four exponents 
 $\gamma_i \geq 3$. If there are precisely  four such exponents, the negation of condition (e) in 
Definition \ref{exceptional} (for $G$) shows that $G_0$ satisfies (iii) or (iv) of Case (4) in 
Proposition \ref{faithful}. If there are precisely 5 such exponents in $G_0$, then (for $G$) the negation 
of  condition (f) in Definition \ref{exceptional} shows that $G_0$ satisfies conditions  (i) or (ii) of 
Case (4) in Proposition \ref{faithful}. Finally, if there are six or more exponents $\gamma_i \geq 3$ in $G_0$, 
then $G_0$ satisfies Case (3) in Proposition \ref{faithful}.

\smallskip

Hence in all  cases the desired representation $\rho: G_0 \to Sl_2(\C)$ is provided by Proposition \ref{faithful}.
\end{proof}



\section{Computations in a cyclic-group ring with coefficients in $\C$}

\label{evaluation}

In this section we will prove Proposition \ref{Lemma-1-7}, which plays a crucial 
role in the proof of Proposition \ref{first-gen} and thus of Theorem \ref{main}.
This section can be read independently from the rest of the paper; the arguments 
presented here include several lengthy 
computations 
in a group ring with complex coefficients.

\smallskip

Throughout this section $p$ and $q$ will denote integers which satisfy 
$$p, q \geq 3 \quad \text{and} \quad p\,|\,q\,,$$  
and we also fix some primitive $q$-th root of unity $\zeta$. Let $t$ be the generator of a c
yclic group $\langle\, t \mid t^p \,\rangle$ of order $p$. For any 
$a \in (\Z /q \Z)^*, \, b \in (\Z /p \Z)^*$ and $r \in \R$ we define the following element in the 
group ring $\C[\langle \, t \mid t^p \, \rangle]$:
$$\Pi(a, b, r) = r(\zeta^a t^a - 1)(\zeta^{-a} t^{-a} - 1)(t^{b} -1)(t^{-b} -1)$$

We compute:
$$\Pi(a, b, r) = r(2 -\zeta^a t^a - \zeta^{-a} t^{-a})(2 -t^{b} -t^{-b})$$
$$= r[2(2 -\zeta^a t^a - \zeta^{-a} t^{-a})-(2 -\zeta^a t^a - \zeta^{-a} t^{-a})t^{b}
-(2 -\zeta^a t^a - \zeta^{-a} t^{-a})t^{-b}]$$
\begin{equation}\label{sum}
= r[(4 -2\zeta^a t^a - 2\zeta^{-a} t^{-a})-(2t^b -\zeta^a t^{a+b} - \zeta^{-a} t^{-a+b})
-(2t^{-b} -\zeta^a t^{a-b} - \zeta^{-a} t^{-a-b})]
\end{equation}

\begin{prop}
\label{Lemma-1-7}
Let $a, b$ and $r$ be as above, and let $a', b'$ and $r'$ be a second such triple. Then
$$\Pi(a, b, r) = \Pi(a',b',r')$$ implies:
$$a = \pm\, a' \in \Z/q\Z$$
\end{prop}

\begin{proof}
This proof proceeds by considering various cases and subcases, where each case needs 
distinct careful considerations. The assumption $p\,|\,q$ implies that $a \in (\Z /q \Z)^*$ 
 has a canonical image $\bar a \in (\Z /p \Z)^*$; however, since below the context is 
 always unambiguous, we will simplify notation and consistently write $a$ for $\bar a$.

\medskip

\noindent {\bf Case 1:} First consider the special case $p = 3$. Then for any $q \in \N$ so that 
$p\,|\,q$ the conditions $a \in (\Z /q \Z)^*$ and $b \in (\Z /p \Z)^*$ imply that  $b = a = \pm 1$ mod $3$, 
or $b = -a = \pm 1$ mod $3$. In both cases we obtain, for $\epsilon = \pm 1\, $:
$$\Pi(a, b, r) = r(\zeta^a t^a - 1)(\zeta^{-a} t^{-a} - 1)(t^{b} -1)(t^{-b} -1) =$$
$$r(4+\zeta^a+\zeta^{-a}) + r(-2-2\zeta^a+\zeta^{-a}) t^{\epsilon} + 
r(-2-2\zeta^{-a}+\zeta^a) t^{-\epsilon}$$
Hence, if  $D$ is the coefficient of the monomial $t^0$, we have the equality
$$r\,  {\rm Re}\, \zeta^{ a} = \frac{1}{2}(D-4r) \, ,$$
and for the other two coefficients $E$ and $F$ we get 
$$r \, {\rm Im}\, \zeta^{a} = \pm \frac{1}{3}(E-F) \, .$$
This gives:
$$r^2 = (\frac{1}{2}(D-4r))^2 + (\frac{1}{3}(E-F))^2 $$

Furthermore, notice that $4+\zeta^a+\zeta^{-a} > 0$ for any value of $a$, so that $r$ and 
$D$ have  the same sign. Hence we can derive the values of  ${\rm Re}\, \zeta^{a}$ and of 
$\pm{\rm Im}\, \zeta^{a}$ from $\Pi(a, b, r)$, and thus also  the value of $\zeta^{\pm a}$. 
This shows that $\Pi(a, b, r) = \Pi(a', b', r')$ implies $a = \pm a'$ in $\Z/q\Z$.

\medskip

\noindent {\bf Case 2:}  Assume from now on that $p, q \geq 4$. Consider the case where
in the expression $\Pi(a, b, r) = \Pi(a',b',r')$ the 9 ``$t$-monomials'' in the sum (\ref{sum}),  
interpreted as ``polynomial'' 
in $t$,  all have distinct exponents.  In other words, the nine exponents
\begin{equation}
\label{9-terms}
0, \,\,a,\,\, -a,\,\, b,\,\, -b,\,\, a+b,\,\, a-b,\,\,-a+b,\,\, -a-b
\end{equation}
define pairwise distinct elements of $\Z/p\Z$. 

Hence $4r$ is the only term in the sum $\Pi = \Pi(a, b, r)$ with $t$-exponent equal to 0. Similarly 
$4r'$ is the only such term in  $\Pi' = \Pi(a',b',r')$. It follows that $r = r'$. Thus, after dividing both 
$\Pi$ and $\Pi'$  by  $r$, we see that $\zeta^a$ and $\zeta^{-a}$ are the only coefficients of any 
$t$-monomial with modulus  $1$ in $\Pi$, and similarly for $\Pi'$. Hence in this case as well we 
can deduce that $a = \pm\, a'$ in $\Z/q\Z$.

\medskip
\noindent {\bf Case 3:}  The remaining case is more complicated and will be dealt 
with by splitting it into various subcases.  First observe that $a \neq - a$ and $b \neq -b$ 
follows from the assumptions $p \geq 4$ and $\gcd(a, p) = \gcd(b, p) =1$. By the same 
argument we deduce that the only cases, where two or more of the nine $t$-exponents, 
listed above in (\ref{9-terms}), can agree,  
are given by: 
\begin{enumerate}
\item $a = \pm b \in \Z/p\Z$, or 
\item $a = \pm 2b \in \Z/p\Z$, or 
\item 
$b = \pm 2a \in \Z/p\Z$, or
\item 
$2a = \pm 2b \in \Z/p\Z$.
\end{enumerate}

We examine now these cases separately:

\smallskip
\noindent (a) Assume $a = b \in \Z/p\Z$ or $a = -b \in \Z/p\Z$. First observe that both 
of these two assumptions exclude (2) and (3), since the relative primeness of $a$ 
and $p$ would imply $p = 3$, contrary to our assumption $p \geq 4$. It is  easily 
checked that  in both cases 
$$\Pi(a, b, r) =$$ 
$$r[(4+\zeta^{-a}+\zeta^a)+( -2\zeta^a - 2)t^a+(- 2\zeta^{-a}  - 2) t^{-a}+
\zeta^a t^{2a}+\zeta^{-a} t^{-2a}]$$

\smallskip
\noindent (b) Assume $b = 2a \in \Z/p\Z$ or $b = -2a \in \Z/p\Z$, and assume $p = 5$, which yields 
$a = -2b \in \Z/p\Z$ or $a = 2b \in \Z/p\Z$ respectively.
We calculate, again for both cases:
$$\Pi(a, b, r) = $$
$$ r[(4 + (-2\zeta^a + \zeta^{-a}) t^{a} +(- 2\zeta^{-a}+\zeta^a) t^{-a}
+(-2 + \zeta^{-a}) t^{2a} +(-2 + \zeta^a) t^{-2a}] $$

\smallskip
\noindent (c) Assume $b = 2a \in \Z/p\Z$ or $b = -2a \in \Z/p\Z$, and assume $p \neq 5$, which 
yields $a \neq \pm 2b \in \Z/p\Z$.
We calculate, again for both cases:
$$\Pi(a, b, r) =$$
$$ r[4 +(-2\zeta^a + \zeta^{-a}) t^a +(- 2\zeta^{-a}+\zeta^a) t^{-a} -2t^{2a} +\zeta^a t^{3a}
-2t^{-2a}  + \zeta^{-a} t^{-3a}] $$

\smallskip
\noindent (d) Assume $a = 2b \in \Z/p\Z$, or $a = -2b \in \Z/p\Z$. If $p = 5$ then we deduce that 
we are back in case (b) above. Thus we can assume $p \neq 5$, which yields $b \neq -2a = -4b \in \Z/p\Z$ 
or $b \neq  2a = -4b \in \Z/p\Z$ respectively. We calculate, again for both cases:
$$\Pi(a, b, r) =$$
$$ r[4 + (\zeta^{2b} -2) t^{b}+ (\zeta^{-2b}  -2) t^{-b} -2\zeta^{2b} t^{2b} - 2\zeta^{-2b} t^{-2b}
+\zeta^{2b} t^{3b} +\zeta^{-2b} t^{-3b}] $$

\smallskip
\noindent 
(e) Assume $2a = 2b \in \Z/p\Z$, or $2a = -2b \in \Z/p\Z$. We can assume that $a \neq \pm b \in \Z/p\Z$, 
as otherwise we are back in case (a). We deduce that $p$ is even, i.e. 
$$p = 2 p'$$
for some integer $p' \geq 1$. It follows that $a = \pm b$ modulo $p'$, and thus $b = \pm a + p' \in \Z/p\Z$. 
As a consequence, from $a \neq \pm b \in \Z/p\Z$ and $\gcd(a, p) = \gcd(b, p) =1$ we deduce that $p \neq 4$ 
and $p \neq 6$, and hence $p \geq 8$. We calculate for  both, $b = a + p' \in \Z/p\Z$ or $b = - a + p' \in \Z/p\Z$ 
that
$$\Pi(a, b, r) =$$
$$r[(4 -2\zeta^a t^a - 2\zeta^{-a} t^{-a}) + (\zeta^{-a} t^{p'-2a} - 2t^{p'-a} + (\zeta^{a}+\zeta^{-a}) t^{p'} 
 - 2t^{p' + a} +\zeta^a t^{p'+2a})].$$
We note that from $p \geq 8$ it follows that all 8 terms in this ``polynomial'' have distinct $t$-exponents.

\medskip

In order to finish this Case 3, we now need to consider the other triple $a', b', r'$; a priori it 
may not fall into the same cases (a) - (e) as the triple $a, b, r$ considered above.

\medskip \noindent (A) 
Assume first that  assumption (e) holds for $a, b, r$. As this is the only case where in the expression $\Pi = \Pi(a, b, r)$ 
there are precisely 8 distinct terms, it follows that the other triple $a', b', r'$ must also be in case (e). This implies, 
by comparing the constant terms, that $r = r'$.  Hence in $\frac{1}{r}\Pi = \frac{1}{r}\Pi'$  the only non-real coefficients 
with modulus 1 are equal to $\zeta^{\pm a}$. Thus we obtain $a = \pm a'$ in $\Z/q\Z$.

\medskip \noindent (B) 
Assume next that  assumptions (c) or (d) hold for $a, b, r$. 
Then $p = 5$ is excluded, and $p \neq 4$ and $p \neq 6$ follow from $a = \pm 2b$ or $b \pm 2a$ 
and the assumption that both, $a$ and $b$, are relatively prime to $p$. Hence one has $p \geq 7$, 
which implies that all 7 terms in the expression of $\Pi = \Pi(a, b, r)$ in the cases (c) and (d) must 
be distinct. It follows that the other triple $a', b', r'$ must also be in cases (c) or (d). As in the previous 
case, by comparing the constant terms we deduce $r = r'$.  And similarly, in $\frac{1}{r}\Pi = \frac{1}{r}\Pi'$ 
the only coefficients with modulus 1 are equal to $\zeta^{\pm a}$, thus showing $a = \pm a'$ in $\Z/q\Z$.

\medskip \noindent (C) We can now assume that both triples $a, b, r$ and $a', b', r'$ are as in 
cases (a) or (b) above. If $p \neq 4$, then for both, (a) and (b), it follows from the fact that $p \neq 3$ 
that all 5 terms in the expression of $\Pi =\Pi'$ are distinct.  In case (a) there are two terms with non-zero 
$t$-exponent, which have the property that their coefficients $E$ and $F$ satisfy $F + 2E \in \R$. This is  
not true for case (b). Thus, either both triples are in case (a), or both are in case (b).

In the first case we notice that for any choice of coefficients $E$ and $F$ with $F + 2E \in \R$ 
one has $\frac{E}{|E|} = \zeta^{\pm a}$, which again yields $a = \pm a'$ in $\Z/q\Z$.
In the second case we can again compute the value $r = r'$ from the only term in 
$\Pi =\Pi'$ with zero $t$-exponent. Further, since $p = 5$ and $\gcd(a, p) = 1$, one has:
$${\rm Re}\,(-2 + \zeta^{\pm a}) < {\rm Re}\,(-2\zeta^a + \zeta^{-a}) = 
{\rm Re}\,(-2\zeta^{-a} + \zeta^{a}) < 4 $$
Hence it suffices to consider the coefficient $E$ of $\Pi = \Pi'$ with the smallest real 
part to observe that $\frac{E+2}{r} = \zeta^{\pm a} =\zeta^{\pm a'}$, which gives again 
$a = \pm a'$ in $\Z/q\Z$.

\medskip 

\noindent (D) It remains to consider the case where both triples $a, b, r$ and $a', b', r'$
are as in cases (a) or (b) above, and in addition we have $p = 4$. The latter, 
however, contradicts the assumption $p = 5$ in (b), so that in fact both triples 
belong to case (a). From $p = 4$ we obtain $t^{2a} = t^{-2a} = t^2$ so that we have 
precisely 4 terms: 
$$\Pi =\Pi' = E + F t + G t^2 + H t^3$$
Now note that $E - G= 4r$, so that one computes 
$1 - \frac{2F}{E-G} = \zeta^{\pm a} =\zeta^{\pm a'}$, which gives once more 
$a = \pm a'$ in $\Z/q\Z$.
\end{proof}



\section{Proof of Theorem \ref{main}}\label{sec:ProofOfMain}

In this section we give a proof of Theorem~\ref{main}. The proof given below applies also to the 
more general statement (see Remark \ref{rem:Generalization}), obtained by weakening the 
assumptions on $n$ and $m$ stated in in the second paragraph of Theorem \ref{main} 
to simply assuming that $G$ is non-exceptional.

The crucial ingredient in this proof is 
Proposition ~\ref{first-gen}. The proof of this proposition is preceded by  a sequence of simplifications
and by three technical lemmas. The proofs of the latter  are deferred to the next section.

\begin{proof}[Proof of Theorem~\ref{main}] 
We first show that the ``if'' direction in the statement of Theorem \ref{main} is a direct consequence of the 
concept of ``Nielsen equivalence'', see Definition \ref{Nielsen-equivalence-M}.

Assume  $j = k$, i.e. the indices $j$ and $k$ of the ``missing generators'' $s_j$ for $\cal U$ and $s_k$ 
for $\cal V$ are identical.  Then the assumption
$$u_i = \pm v_i \quad \text{modulo} \quad \gamma_i$$
implies that $\cal U$ and $\cal V$ are the same up to inversion of some generators, which is 
one of the allowed operations within a Nielsen equivalence class. 

Assume $j \neq k$. By assumption we have $1 = u_j = v_j$, so that $s_j$ is part of the family $\cal V$. 
We apply to the generating system $\cal V$ the operation which replaces $s_j$ first by $s_j^{-1}$, and 
then the latter by $s_k = (s_{k-1}^{-1} \ldots s_{1}^{-1}  s_{\ell}^{-1} \ldots s_{j+1}^{-1}) \cdot s_j^{-1} 
\cdot (s_{j-1}^{-1} \ldots s_{k+1}^{-1})$, if $k < j$.  For $j < k$ use the analogous operation. Such 
replacements are all Nielsens operations, and the result is a standard generating system $\cal V'$ 
with the same ``missing generator'' as $\,\cal U$. Hence the arguments for the above treated case $j = k$ 
apply, to conclude that $\cal U$ is Nielsen equivalent to $\cal V'$ and hence to $\cal V$.

\smallskip

The``only if'' statement of Theorem \ref{main} follows from Proposition \ref{first-gen} stated below: 
Since the generators  in the  presentation (\ref{Fuchs}) of $G$ can be permuted by formula 
(\ref{gernerator-permutation}), we may restrict our attention to the first standard generator $s_1$,  
and then repeat the argument for the other $\ell -1$ standard generators.

It has been already verified in Remark \ref{exceptional-quotient} (2) that the assumption  
used in Proposition \ref{first-gen}, that $G$ is non-exceptional, is weaker than the 
assumptions $m \geq 5$ for even $n$, and $m \geq 7$ if $m$ is odd, from Theorem \ref{main}.
\end{proof}

\begin{prop} \label{first-gen}
Let $G$ be a group with a presentation as in (\ref{Fuchs}), and assume that $G$ is non-exceptional 
(see Definition \ref{exceptional}). Let $\mathcal{U}$ and $\mathcal{V}$ be two standard generating 
systems for $G$ as defined in  Definition \ref{def:scheme}. If  $\,\cal{U}$ and $\mathcal{V}$ are Nielsen 
equivalent, then one has
$$u_1 = \pm v_1 \quad \text{modulo} \quad \gamma_1 \, .$$
\end{prop}

 Before the proof of this proposition is presented, we go through some preliminary considerations.
First note that without loss of generality  the hypotheses on $G$ can be  strengthened slightly by 
making use of the following observation:

\begin{lem} \label{even-4}
Consider the special case where for every standard generator $s_i$ of $G$ the exponent $\gamma_i$ 
satisfies the following: If $\gamma_i$ is even and $\gamma_i \neq 2$,  then $\gamma_i$ is divisible by 
4 (in other words: $G = G^\#$). 

Then proving Proposition \ref{first-gen} for any $G$  as in  this special case implies 
 Proposition \ref{first-gen} in full generality.
\end{lem}
 
\begin{proof}
By assumption, $G$ has a presentation as in (\ref{Fuchs}).  For every even exponent 
$\gamma_i \neq 2$ which is not divisible by 4  define a new exponent 
$\hat\gamma_i = \frac{\gamma_i}{2}$ (which is an odd integer), and consider the 
canonical 4-quotient group 
$$G^\# = G/\langle \langle \{ s_i^{\hat \gamma_i} \}\rangle \rangle \, ,$$
 as in Definition \ref{2-quotient} (2). We know from Remark \ref{exceptional-quotient} (1) that, 
 if $G$ is non-exceptional, then so is $G^\#$.

Assume that the generating  systems $\mathcal{U}$ and $\mathcal{V}$ are Nielsen equivalent. 
Then (see Remark \ref{quotient-Nielsen})  their images in the quotient group $G^\#$ are also 
Nielsen equivalent. If Proposition \ref{first-gen} holds for $G^\#$, then we know that 
$u_1 = \pm v_1$ modulo $\hat \gamma_1$.  If $\gamma_1 \neq \hat \gamma_1$, then  
modulo $\gamma_1$ we have:
$$u_1 = \pm v_1 \quad \text{or} \quad  u_1 = \pm v_1 + \hat \gamma_1$$
However, since $\hat \gamma_1$ is odd, it follows that for any integer $k$ at most one of $k$ or 
$k + \hat \gamma_1$ can be relatively  prime to $\gamma_1 = 2 \hat \gamma_1$. Since by 
Definition \ref{def:scheme} both, $u_1$ and $v_1$, are assumed to be relatively prime to 
$\gamma_1$, we deduce that
$$u_1 = \pm v_1 \quad \text{modulo} \quad \gamma_1 \, .$$ 
\end{proof}

Next, observe that in Proposition \ref{first-gen} we can assume 
\begin{equation}\label{first-exponent}
\gamma_1 \geq 5 
\end{equation}
since for  $\gamma_1 \leq 4$ the conclusion of the proposition becomes trivial.

\smallskip

Furthermore, in the special case that $\gamma_1$ is relative prime to all other $\gamma_i \geq 3$, 
the proof of Proposition \ref{first-gen} becomes much simpler, as will be seen below. The complementary 
case, though, poses several problems, which are dealt with now, using the work already done in the 
previous sections.

\smallskip

We thus assume from now on that $\gamma_1$ is not relatively prime to some other $\gamma_i \geq 3$.  
Then we can assume further from the commutator equality (\ref{gernerator-permutation}) that 
 $i = 2$, and from the extra hypothesis  $G = G^\#$, achieved in Lemma \ref{even-4}, that 
some integer $p \geq 3$ is a common divisor of $\gamma_1$ and $\gamma_2$. 

\smallskip

From Lemma \ref{reps} (2) we know that there exists a cyclic-faithful representation
$$\eta'_1:  G_1 = G / \langle \langle s_2 \rangle \rangle \to Sl_2(\C) \, .$$
In particular,  every generator $s_i$  other than $s_2$ 
is mapped by $\eta'_1$, up to conjugation in $Sl_2(\C)$, to a primitive $\gamma_i$-matrix 
\begin{equation}
\label{M-zeta}
M(\zeta_i) = 
\left[
\begin{array}{cc}
\zeta_i & 0\\
0 & \zeta_i^{-1}
\end{array} \right] \, .
\end{equation}
More specifically, after possibly conjugating $\eta'_1$ in $Sl_2(\C)$, we can require that 
$\eta'_1(s_1) = M(\zeta_1)$, while for $i \geq 3$ we only require that $\eta'_1(s_i)$ and $M(\zeta_i)$ 
agree up to conjugation in $Sl_2(\C)$.

Let $\eta_1:G\to Sl_2(\C)$ be the composition of the quotient map $G \to G / \langle \langle s_2 \rangle \rangle$ 
with $\eta'_1$. Consider now the quotient homomorphism 
$$\eta_2: G \to \langle t \mid t^p \rangle, \,\,
s_1 \mapsto t, \,\, s_2 \mapsto t^{-1}, \,\, s_i \mapsto 1  \,\, (i \geq 3)$$
and combine the maps $\eta_1$ and $\eta_2$ to obtain a homomorphism
$$\eta: G \to Sl_2(\C[\langle t \mid t^p \rangle])$$
given by $\eta(s_i) = \eta_1(s_i)$ for $i \geq 3$, and by
$$\eta(s_1) = \left[
\begin{array}{cc}
\zeta_1 t & 0\\
0 & \zeta_1^{-1} t^{-1}
\end{array} \right]
\quad
\text{and}
\quad
\eta(s_2) = 
\left[
\begin{array}{cc}
t^{-1} & 0\\
0 & t
\end{array} \right]\, .$$

\medskip

Consider now the generating system 
$$\cal U = (x_1 = s_1^{u_1}, \ldots, x_{j-1} = s_{j-1}^{u_{j-1}}, x_{j+1} = s_{j+1}^{u_{j+1}}, 
\ldots, x_{\ell} = s_{\ell}^{u_{\ell}}),$$ 
and recall that for each generator $x_i$ the exponent $u_i$ is relatively prime to $\gamma_i$ 
(and that this holds also for the formally introduced exponent $u_j = 1$). 
Thus we  can  pick an integer $z_i \in \Z$ with 
\begin{equation}\label{exponents}
 z_i \cdot u_i = 1\,\,\,  \text{mod}\,\, \gamma_i\,,
\end{equation}
and obtain in $G$ the equalities 
$$s_i = x_i^{z_i} \quad \text{for} \quad x_i = s_i^{u_i} \, .$$
For a family $X =(X_1, \ldots, X_{j-1}, X_{j+1}, \ldots,  X_\ell )$ of formal symbols $X_i$ consider 
as in (\ref{canonical-surjection}) the free group $F(X)$ and the canonical surjection 
$$p_\cal U: F(X) \twoheadrightarrow G, \,\, X_i \mapsto x_i = s_i^{u_i}\, .$$

For any second generating system ${\mathcal{W}} = (w_1, \ldots, w_{\ell-1})$ of $G$ each 
element $w_h$ can be written as a word in the $x_i^{\pm1}$ with $i \neq j$, and hence we 
obtain a family $W$ of elements $W_h \in F(X)$ which satisfy 
\begin{equation}
\label{lift-M}
p_\cal U(W_h) = w_h \, .
\end{equation}

As has been discussed in Section \ref{prelims} (see Proposition \ref{Jacobian-M}), compute 
the $(\ell-1) \times (\ell -1)$-matrix $\partial W/ \partial X$ of Fox derivatives $\partial W_h / \partial X_i \in \Z F(X)$.
We denote by $\partial {\mathcal{W}}/\partial {\mathcal{U}}$ the associated Jacobian matrix in the matrix ring 
$\mathbb{M}_{(\ell-1) \times (\ell -1)}(\Z G)$, i.e. $\partial {\mathcal{W}}/\partial {\mathcal{U}}$ is the image of 
$\partial W/ \partial X$ under the map induced by $p_\cal U$.

In order to apply the method from Corollary \ref{evaluation-M}, we now pass to the image 
of $\partial {\mathcal{W}}/\partial {\mathcal{U}}$ under the above defined ``mixed'' representation 
$\eta$, and compute the determinant $D({\mathcal{W}}, {\mathcal{U}})$ of the resulting matrix 
$M({\mathcal{W}},{\mathcal{U}}) = \eta(\partial {\mathcal{W}}/\partial {\mathcal{U}}) \in 
Sl_{2\ell - 2}(\C[\langle t \mid t^p \rangle])$. 
As a final step, we multiply $D({\mathcal{W}}, {\mathcal{U}})$ by the product $\Pi(u_1, u_2, 1)$, 
where we use the expression
\begin{equation}
\label{Pi}
\Pi(a, b, r) := r(\zeta_1^a t^a - 1)(\zeta_1^{-a} t^{-a} - 1)(t^{b} -1)(t^{-b} -1)
\end{equation}
defined in Section \ref{evaluation},
with 
parameters specified to $a =u_1, b = u_2$ and $r = 1$.

\begin{rem}\label{lapsus}
We should alert the reader that the above introduced notation $D({\mathcal{W}}, {\mathcal{U}})$ 
is slightly misleading, since the value of this determinant may well depend not just on $\cal U$ and 
$\cal W$, but also on the chosen lifts $W_h$ of the elements $w_h \in \cal W$. However, 
Lemma \ref{independent} below ``repairs'' this lapsus, which mainly serves to avoid 
adding further extra notation.
\end{rem}

We now state three lemmas which will be proved in the next section. We then show that combining 
these three lemmas and applying Proposition \ref{Lemma-1-7} yields, without much ado, the statement of 
Proposition \ref{first-gen}.

\begin{lem} 
\label{independent}
The product
$$\Pi(u_1, u_2, 1) D({\mathcal{W}}, {\mathcal{U}}) =$$
$$(\zeta_1^{u_1} t^{u_1} - 1)(\zeta_1^{-{u_1}} t^{-{u_1}} - 1)(t^{{u_2}} -1)(t^{-{u_2}} -1)
 \det(\eta(\partial {\mathcal{W}}/\partial {\mathcal{U}}))$$
depends only on the families $\mathcal{W}$ and $\mathcal{U}$ in $G$, and not on the particular 
choice  of the words $W_h$ in the free group $F(X)$ which represent via (\ref{lift-M}) the elements 
of the generating system $\mathcal{W}$.
\end{lem}

\begin{lem}
\label{trivial-class}
If $\mathcal{U}$ and $\mathcal{W}$ are Nielsen equivalent, then one obtains:
$$\Pi(u_1, u_2, 1) D({\mathcal{W}}, {\mathcal{U}}) = \Pi(u_1, u_2, 1)$$
\end{lem}

\begin{lem}\label{result}
Given  generating systems $\mathcal{U}$ and $\mathcal{V}$ of $G$ as  in Definition \ref{def:scheme} 
one computes

$$\Pi(u_1, u_2, 1) D({\mathcal{V}}, {\mathcal{U}}) = \Pi(v_1, v_2, r)\, ,$$
for some value $r \in \R$.
\end{lem}

\begin{proof} [ Proof of Proposition~\ref{first-gen}]  

(1) First consider the  case,  treated above, where $\gamma_1$ is not relatively prime to some other 
$\gamma_i \geq 3$. As shown above, we can assume $i=2$ and $p = \gcd(\gamma_1, \gamma_2) \geq 3$.

While (as pointed out in Remark \ref{lapsus}) the determinant $D({\mathcal{V}}, {\mathcal{U}})$ may well 
depend on the choice of the lifts under the map $p_\cal U$ of the elements $v_i \in \cal V$, it follows from 
Lemma ~\ref{independent} that the product $\Pi(u_1, u_2, 1) D(\mathcal{V}, \mathcal{U})$ is a true invariant 
of the two generating systems $\cal U$ and $\cal V$ of $G$. Hence combining Lemmas ~\ref{trivial-class}  
and  ~\ref{result}  allows us to conclude, for Nielsen equivalent generating systems $\cal U$ and $\cal V$ 
as in Proposition \ref{first-gen}, that $\Pi(v_1, v_2, r) = \Pi(u_1, u_2, 1)$.  Now apply  Proposition ~\ref{Lemma-1-7}, 
for $q = \gamma_1$, $(a, b, r) = (v_1, v_2, r)$ and $(a', b', r') = (u_1, u_2, 1)$, to  directly obtain the conclusion of 
Proposition ~\ref{first-gen}.

\smallskip
\noindent (2)
Let us now assume that $\gamma_1$ is relatively prime to all other $\gamma_i \geq 3$.

Then $G$ has, by Lemma \ref{even-4} and Lemma \ref{reps} (1), a representation in  
$Sl_2(\C)$ which is faithful on every cyclic subgroup that is generated by one of the generators 
$s_i$. But then Proposition \ref{first-gen} is a direct consequence of what has been shown in 
previous work of the authors, see \cite{LuMo:NEinFuch}, Lemma 1.9.  Indeed, all 
the arguments used in Lemma 1.9 of \cite{LuMo:NEinFuch} are based on the fact that, under the 
conditions given in this lemma, there is a cyclic-faithful representation of $G$ in $Sl_2(\C)$.
\footnote{\, The evaluation methods employed in \cite{LuMo:NEinFuch} are quite different from 
the techniques described explicitly in this and the previous sections. As a consequence,  
a quick inclusion of the  methods from \cite{LuMo:NEinFuch},  for the benefit of completeness of the 
presentation here, doesn't  quite seem possible.}
\end{proof}
 
\begin{rem}\label{N-torsion-question}
In their previous work \cite{LuMoGenSysTor} the authors have defined the {\em Nielsen torsion} 
$\cal N(\cal V, \cal U)$,  for any minimal generating systems $\cal U$ and $\cal V$ of a finitely 
generated group $G$. This torsion invariant depends only on the Nielsen equivalence classes of 
$\cal U$ and $\cal V$, and it is based on the same Fox derivative approach as used here.

The  invariant $\cal N(\cal V, \cal U)$  is an element in the  first $K$-group $K_1(\Z G/I_G)$ over the quotient 
of the group ring $\Z G$ modulo the Fox ideal $I_G$.  Here $I_G$ is the two-sided ideal generated by the 
$p_\cal U$-images of the  Fox derivatives $\partial R/\partial X_i$, for any $R \in \ker(p_\cal U: F(X) \twoheadrightarrow G)$ 
and any element $X_i$ of $X = (X_1, \ldots, X_n)$. More precisely, $\cal N(Y, X)$ lies in the quotient 
(called $\cal N(G)$) of $K_1(\Z G/I_G)$ modulo the subgroup $T_G$ of all trivial units, i.e. all elements given 
by $\pm g$ for  any $g \in G$.

A careful analysis of the proof of Proposition \ref{first-gen} presented in this section reveals that, for any 
two standard generating systems $\cal U, \cal V$ of a non-exceptional Fuchsian group $G$, one has actually
$$\cal N(\cal V, \cal U) \neq 1 \, ,$$ 
if the the family of exponents for $\cal U$ and $\cal V$ do not satisfy the condition $u_i = \pm v_i$ 
modulo $\gamma_i$, for all $i = 1, \ldots, \ell$. 

This is in fact a stronger statement than the one given in Theorem \ref{main}, since there are pairs of 
minimal generating systems (in different groups $G$) which are known to be not Nielsen equivalent, 
but have trivial $\cal N$-torsion. Since $\cal N$-torsion behaves functorially (see Theorem I (iv) of 
\cite{LuMoGenSysTor}), this can be used to exhibit inequivalent generating systems in certain quotients 
of $G$, while in general Nielsen inequivalence is {\em not} preserved when passing even to mild quotients 
of a group. 
\end{rem}
\bigskip


\section{Proof of three lemmas} \label{proof-of-lemmas}
\bigskip
 It remains to prove  Lemmas \ref{independent}, \ref{trivial-class} and \ref{result}. We will 
use the notation and terminology introduced in the previous section.

\begin{proof}[{\bf Proof of Lemma \ref{independent}}]
Any second set $W^* \subset F(X)$ of lifts of the elements in $\mathcal{W}$ under 
the map $p_\cal U$ gives rise to a second ``Jacobian  matrix'' 
$M({\mathcal{W}^*},{\mathcal{U}}) \in Sl_{2\ell - 2}(\C[\langle t \mid t^p \rangle])$ analog to 
$M({\mathcal{W}},{\mathcal{U}})$. It satisfies (see Proposition \ref{Jacobian-M})
$$M({\mathcal{W}^*}, {\mathcal{U}}) = M({\mathcal{W}},{\mathcal{U}}) + A\, ,$$
where the matrix $A$ has the property that each row  is given by the $\eta$-image 
of some $(\ell-1)$-tuple
$$(\partial R/\partial X_1, \ldots, \partial R/\partial X_{j-1}, \partial R/\partial X_{j+1}, 
\ldots, \partial R/\partial X_\ell)\, ,$$  
with $R \in \ker(p_\cal U: F(X) \to G)$.

In particular, if $\ker p_\cal U$ is normally generated by elements $R_1, \ldots, R_t$, then each coefficient 
of $A$ is the $\eta$-image of a sum of $\Z G$-left-multiples of $p_\cal U(\partial R_s/\partial X_i)$, with 
$i \in \{1, \ldots, j-1, j+1, \ldots, \ell\}$ and $s \in \{1, \ldots, t\}$.  As a consequence (see Corollary \ref{evaluation-M}), 
the determinant $D({\mathcal{W}^*}, {\mathcal{U}})$, analogously defined as the determinant 
$D({\mathcal{W}}, {\mathcal{U}})$ before Remark \ref{lapsus}, satisfies the equality
$$D({\mathcal{W}^*}, {\mathcal{U}}) =  D({\mathcal{W}}, {\mathcal{U}}) +  B\, ,$$
where $B \in \C[\langle t \mid t^p \rangle]$ is a sum of products which all contain, as factor, a coefficient 
of one of the above $(2\times 2)$-matrices $(\eta \circ p_\cal U)(\partial R_s/\partial X_i)$. Hence the claim 
of Lemma \ref{independent} follows if we prove
\begin{equation}
\label{b-test}
\Pi(u_1, u_2, 1) \, b = 0
\end{equation}
for any such coefficient $b$.

\smallskip

Observe  (by performing a suitable sequence of Tietze operations on the presentation (\ref{Fuchs}) of 
$G$) that  the kernel of  the surjection $p_\cal U: F(X) \twoheadrightarrow G, \,X_i \mapsto x_i = s_i^{u_i}$ 
is normally generated by the elements 
$$X_i^{\gamma_i} \quad \text{for all} \quad i\in \{1, \ldots, j-1, j+1, \ldots, \ell\}\, ,$$
together with the relator 
$$R_0 = (X_{j+1}^{z_{j+1}} \ldots X_{\ell}^{z_{\ell}} X_{1}^{z_{1}} \ldots X_{j-1}^{z_{j-1}})^{\gamma_j} \, ,$$
for the exponents $z_i$ as defined in (\ref{exponents}).

Now $\partial X_i^{\gamma_i} / \partial X_h = 0$ for $h \neq i$, and 
$\partial X_i^{\gamma_i} / \partial X_i = 1 + X_i + \dots + X_i^{\gamma_i -1}$. For $i \geq 3$ the 
$(\eta \circ p_\cal U)$-image of $X_i$ is conjugate to the matrix $M(\zeta_i)$ as in (\ref{M-zeta}), 
so that since
\begin{equation}
\label{zetas}
1 + \zeta_i + \zeta_i^2 + \ldots + \zeta_i^{\gamma_i -1} = 0
\end{equation}
we have $(\eta \circ p_\cal U)(\partial X_i^{\gamma_i} / \partial X_i) = 
I_2 + M(\zeta_i) + M(\zeta_i^2) + \ldots + M(\zeta_i^{\gamma_i -1}) = 0$.
Note that this argument is also true for the special 
case $\gamma_i = 2$, see Remark \ref{element-lift} (3).

\smallskip

For $i = 2$ the matrix $(\eta \circ p_\cal U)(\partial X_2^{\gamma_2} / \partial X_2)$ is conjugate to
$$\left[
\begin{array}{cc}
\Sigma_0 & 0\\
0 & \Sigma_1
\end{array} \right]$$
with $\Sigma_0 = 1 + t^{-u_2} + (t^{-u_2})^2 + \ldots + (t^{-u_2})^{\gamma_2 -1} = 1 + t + \dots + t^{\gamma_2 - 1}$.
Since $p$ is a divisor of $\gamma_2$, we have $(t^{u_2} - 1) \Sigma_0 = 0$, and thus
$$\Pi(u_1, u_2, 1) \, \Sigma_0 = 0 \, .$$
The analogous calculations shows $\Pi(u_1, u_2, 1) \, \Sigma_1 = 0 \, .$

\smallskip
For $i = 1$ the situation is similar: One obtains
$$(\eta \circ p_\cal U)(\partial X_1^{\gamma_1} / \partial X_1) = \left[
\begin{array}{cc}
\Sigma'_0 & 0\\
0 & \Sigma'_1
\end{array} \right]$$

\noindent with $\Sigma'_0 = 1 + \zeta_1 t^{u_1} + (\zeta_1 t^{u_1})^2 + \ldots + 
(\zeta_1 t^{u_1})^{\gamma_1 -1}$,
which gives
$$(\zeta_1 t^{u_1} - 1) \Sigma'_0 = \Pi(u_1, u_2, 1) \, \Sigma'_0= 0 \, .$$
For $\Sigma'_1$ the computations are  essentially the same.

\smallskip

It remains to check the relator $R_0 = (X_{j+1}^{z_{j+1}} \ldots X_{\ell}^{z_{\ell}} X_{1}^{z_{1}} \ldots X_{j-1}
^{z_{j-1}})^{\gamma_j}$. 
\smallskip
Use the chain rule for Fox-derivatives 
(see (\ref{ChainRule}))
and the abbreviation 
$X_j = X_{j+1}^{z_{j+1}} \ldots X_{\ell}^{z_{\ell}} X_{1}^{z_{1}} \ldots X_{j-1}^{z_{j-1}}$
 (recalling $i \neq j$), to compute
\begin{equation}
\label{relation-R}
\partial R_0 / \partial X_i = (1 + X_j + X_j^2 + \ldots + X_j^{\gamma_j -1}) \,  \partial X_j / \partial X_i \, .
\end{equation}
However, $p_\cal U$ maps $X_j$ to $s_j^{-1}$, which in turn is mapped by $\eta$ to a conjugate of the matrix 
$M(\zeta_j^{-1})$. Hence we are now able to apply the same argument as above for the relators $X_i^{\gamma_i}$, 
as follows: 

For the case $j \geq 3$ one computes directly, from equality 
(\ref{zetas}) 
with $j$ replacing $i$, 
that (\ref{relation-R}) gives
(for any index 
$i \neq j$) 
$$(\eta\circ p_\cal U)(\partial R_0 / \partial X_i) = 0 \, .$$

For $j = 1$ or $j = 2$ every coefficient of $(\eta\circ p_\cal U)(\partial R_0 / \partial X_i)$ is the sum of 
products  each of which contains as factor one of the terms $\Sigma_0, \Sigma_1, \Sigma'_0$ or $\Sigma'_1$ 
defined above. In this case we have shown already that multiplication with $\Pi(u_1, u_2, 1)$ annihilates 
each such sum.

\smallskip

Thus  the equality (\ref{b-test}) holds for any coefficient $b$ as desired, and hence the claim stated in 
Lemma \ref{independent} is  proved.
\end{proof}

\smallskip

\begin{proof} [\bf Proof of Lemma \ref{trivial-class}]
For any generating system $\mathcal{W}$ of $G$ we know from Lemma \ref{independent} that 
the left hand side of the equality claimed in Lemma \ref{trivial-class} doesn't depend on the 
choice of the lift $W$ of $\mathcal{W}$ under map $p_\cal U: F(X) \to G$. 

By Theorem \ref{Nielsen-M} we can use the assumption that $\mathcal{W}$ is Nielsen equivalent 
to $\mathcal{U}$ to pick such a lift $W \subset F(X)$ which is a basis of $F(X)$. It follows (see 
Proposition \ref{decomposition-M}) that the matrix $\partial {\mathcal{W}}/ \partial {\mathcal{U}}$ is a 
product of generalized elementary $\Z G$-matrices. Hence $D({\mathcal{W}}, {\mathcal{U}})$ is the 
product of the determinants of the $\eta$-images of these elementary matrices, and thus a product of 
terms of type
$$\det \eta(\pm g) \quad \text{ \rm with} \quad g \in G \, .$$
However, from the  definition of $\eta$ in section \ref{sec:ProofOfMain} we  compute directly that 
 $\det \eta(s_i) = \det \eta(-s_i) = 1$ for any $i \in \{1, \ldots, \ell\}$. This proves the claim of 
 Lemma \ref{trivial-class}.
\end{proof}

\smallskip


\begin{proof} [\bf Proof of Lemma \ref{result}]
Consider the generating system 
$${\mathcal{V}} =(y_1= s_1^{v_{1}}, \dots, y_{k-1} = s_{k-1}^{v_{k-1}}, y_{k+1} = s_{k+1}^{v_{k+1}},
 \dots,  y_\ell = s_\ell^{v_\ell} )$$
 for $G$.  For $h \notin \{j, k\}$  define the element $Y_h = X_h^{z_h v_h} \in F(X)$,
where $X_h \in X$, and the $z_h$ are given in (\ref{exponents}). 
 Recall the formal definitions $u_j = v_k = 1$ and set $y_k = s_k$.

\smallskip

Compute now (recalling $x_i = s_i^{u_i}$)
$$s_j = (s_{j+1} \dots s_\ell \, s_1 \dots s_{j-1})^{-1} = (x_{j+1}^{z_{j+1}}
 \dots x_{\ell}^{z_{\ell}} x_{1}^{z_{1}} \dots x_{j-1}^{z_{j-1}})^{-1}$$
and set
$$Y_0 = (X_{j+1}^{z_{j+1}} \ldots X_{\ell}^{z_{\ell}} X_{1}^{z_{1}} \ldots X_{j-1}^{z_{j-1}})^{-1}$$
as well as 
$$Y_k = Y_0^{v_j} = (X_{j+1}^{z_{j+1}} \ldots X_{\ell}^{z_{\ell}} X_{1}^{z_{1}} \ldots X_{j-1}^{z_{j-1}})^{-v_j} \, .$$
This gives:
$$p_\cal U(Y_h) = y_h \,\,\, \text{for}\,\,\, h \notin \{j, k\}\, , \quad \text{and} \quad 
p_\cal U(Y_0) = s_j, \,\,\, p_\cal U(Y_k) = y_j \, .$$
Now compute the $(\ell-1) \times (\ell -1)$-matrix $\partial Y/ \partial X$ of Fox derivatives 
$\partial Y_h / \partial X_i$ with $h, i \neq j$, and denote by $\partial {\mathcal{V}}/\partial {\mathcal{U}}$ 
its image in the matrix ring  $\mathbb{M}_{(\ell-1) \times (\ell -1)}(\Z G)$, under the map induced by $p_\cal U$.

\smallskip

In order to understand the matrix $\partial {\mathcal{V}}/\partial {\mathcal{U}}$, compute for 
$h \notin \{j, k\}$ the Fox derivatives
\begin{equation}
\label{diago}
\partial Y_h/\partial X_h = 1 + X_h + X_h^2 + \ldots + X_h^{z_h u_h -1}
\end{equation}
and
$$\partial Y_h/\partial X_i = 0 \quad \text{for} \quad i \neq h, k \, .$$
Furthermore, using the formula  (\ref{inverse-formula}) we obtain:
\begin{equation}
\label{diago-k}
\partial Y_k/\partial X_i = 
(1 + Y_0 + Y_0^2 + \ldots + Y_0^{v_j -1})\, \partial Y_0/\partial X_i
\end{equation}
$$ = - (1 + Y_0 + Y_0^2 + \ldots + Y_0^{v_j -1})\, Y_0 \, \partial (X_{j+1}^{z_{j+1}}
 \dots X_{\ell}^{z_{\ell}} X_{1}^{z_{1}} \ldots X_{j-1}^{z_{j-1}})/\partial X_i $$
$$ = - (1 + Y_0 + Y_0^2 + \ldots + Y_0^{v_j -1})\, (X_{j-1}^{-z_{j-1}} 
\ldots X_i^{-z_{i}}) \, \partial X_{i}^{z_{i}}/\partial X_i $$

\smallskip

It follows that for $j = k$ the matrix 
$\partial {\mathcal{V}}/\partial {\mathcal{U}} = (p_\cal U(\partial Y_h/\partial X_i)_{h, i \in \{1, \ldots, j-1, j+1, \ldots, \ell\}}$ 
is a diagonal matrix, while for $j \neq k$ it differs from a diagonal matrix only in the line with index $h = k$. 
In both cases, if we now apply the representation $\eta$  to obtain the matrix 
$M( {\mathcal{V}}, {\mathcal{U}})$, then its determinant $D( {\mathcal{V}}, {\mathcal{U}})$ is 
the product of the determinants of the $(2 \times 2)$-diagonal blocks $M_h$ of $M( {\mathcal{V}}, {\mathcal{U}})$.
Hence the equality claimed in Lemma \ref{result} is equivalent to proving the following equality:
$$[(\zeta_1^{u_1} t^{u_1} - 1)(\zeta_1^{-{u_1}} t^{-{u_1}} - 1)(t^{{u_2}} -1)(t^{-{u_2}} -1) ] \cdot
(\det M_1 \cdot \ldots \cdot\det M_{j-1} \cdot \det M_{j+1} \cdot \ldots \cdot \det M_\ell)$$
\begin{equation} \label{result-M}
= r(\zeta_1^{v_1} t^{v_1} - 1)(\zeta_1^{-{v_1}} t^{-{v_1}} - 1)(t^{{v_2}} -1)(t^{-{v_2}} -1) 
\end{equation}
for some $r \in \R$.

In order to prove (\ref{result-M}) we now evaluate the $(2 \times 2)$-matrix $M_h = \eta(\partial Y_h/\partial X_h)$ 
in the various possible cases for the indices $j, k$ and $h$, where we keep in mind that one always has
$$h \neq j \, .$$ 

\medskip

\noindent \underline{(A) Assume $h \geq 3$ and $j \geq 3$:}

\medskip

For the case $h \neq k$ we observe from (\ref{diago}) that $M_h$ is conjugate to a diagonal matrix 
with complex-conjugate terms in the diagonal. Thus we have
$$\det M_h \in \R\, ,$$
so that its value doesn't effect the equality claimed in (\ref{result-M}).

For the case that $h = k$ we obtain from (\ref{diago-k}) that $\det M_h$ is the product 
\begin{equation}\label{III}
\det M_h = \text{I} \cdot \text{II} \cdot \text{III}
\end{equation}
of the determinants of three types of matrices, namely: 

\begin{enumerate}
\item[I] 
$= \det(-\eta(1 + s_j + \dots s_j^{v_j - 1}))$, 
\item[II] 
$= \det(\eta(s_{j-1}^{-1} \dots s_1^{-1} s_\ell^{-1} \dots s_i^{-1}))$, and \item[III]
$= \det(\eta(1 + s_h^{u_h} + \dots + s_h^{u_h(z_h - 1)}))$. 
\end{enumerate}
Independently of the choice of the indices the determinant of type II is always contained in $\R$.
The same is true for the determinants of type I and III, as long as we assume,  as in the present 
case (A), that $j \geq 3$ and  $h \geq 3$. In case (B) below the product decomposition (\ref{III}) of 
$\det M_h$ is still true, but the factors I or III will take on non-real values.

\medskip
\noindent
\underline{(B) Assume $h \leq 2$ or $j \leq 2$:}
\medskip

 Case (B) will be split below into 8 subcases (a) - (h). In each of them we will apply an argument similar 
to the one that  has already been used in the proof of Lemma \ref{independent}. In order to simplify 
the exposition, we use, for any integer $q \geq 1$, the notation
$$\Sigma_q = 1 + \zeta_1 t^{u_1} + \ldots + \zeta_1^{q-1} (t^{u_1})^{q-1} \in \C[\langle t \mid t^p \rangle]$$
$$\Sigma'_q = 1 + t^{u_2} + \ldots + (t^{u_2})^{q-1} \in \Z[\langle t \mid t^p \rangle] \, ,$$
and observe that
\begin{equation} \label{clue}
(\zeta_1^{u_1} t^{u_1} - 1) \Sigma_q = (\zeta_1^{q u_1} t^{q u_1} - 1)
\quad \text{and} \quad
(t^{u_2} - 1) \Sigma'_q = (t^{q u_2} - 1) \, .
\end{equation}

Now the eight remaining cases are considered:

\medskip

\noindent $\bf{(a)}$ $h = 1$ and $k \neq 1$:
One has \,\, 
$\det M_1 = \Sigma_{z_1 v_1} \cdot \bar \Sigma_{z_1 v_1}$ (where $\bar \Sigma_q$ denotes 
 the complex-conjugate of $\Sigma_q$). Hence (\ref{clue}) gives:
$$ (\zeta_1^{u_1} t^{u_1} - 1) (\zeta_1^{-u_1} t^{-u_1} - 1) \det M_1 = $$
$$(\zeta_1^{z_1 v_1 u_1} t^{z_1 v_1 u_1} - 1) (\zeta_1^{-z_1 v_1 u_1} t^{-z_1 v_1 u_1} - 1) =
(\zeta_1^{v_1} t^{v_1} - 1) (\zeta_1^{-v_1} t^{-v_1} - 1) $$

\bigskip

\noindent $\bf(b)$ $h = 2$ and $k \neq 2$:
One has 
\,\,$\det M_2 = \Sigma'_{z_2 v_2} \cdot \bar \Sigma'_{z_2 v_2}$.
Hence (\ref{clue}) gives:
$$(t^{u_2} - 1) (t^{-u_2} - 1) \det M_2 =$$ $$ (t^{z_2 v_2 u_2} - 1) (t^{-z_2 v_2 u_2} - 1) = (t^{v_2} - 1) (t^{-v_2} - 1)$$

\bigskip

\noindent $\bf{(c)}$ $h = 1$, $k = 1$ and $j \geq 3$: 
In this case the determinant decomposition (\ref{III}) of $\det M_1$ has real factor $\text{I}$ but non-real 
factor $\text{III}$, which gives $\det M_1 = r_0 \Sigma_{z_1 v_1} \bar \Sigma_{z_1 v_1}$ for some $r_0 \in \R$. 
Hence (\ref{clue}) gives:
$$(\zeta_1^{u_1} t^{u_1} - 1) (\zeta_1^{-u_1} t^{-u_1} - 1) \det M_1 = $$
$$r_0 (\zeta_1^{z_1 v_1 u_1} t^{z_1 v_1 u_1} - 1) (\zeta_1^{-z_1 v_1 u_1} t^{-z_1 v_1 u_1} - 1)
= r_0 (\zeta_1^{v_1} t^{v_1} - 1) (\zeta_1^{-v_1} t^{-v_1} - 1)$$

\bigskip

\noindent$\bf{(d)}$ $h = 2$, $k = 2$ and $j \geq 3$:  Again, the determinant decomposition (\ref{III}) of $\det M_2$ 
has real factor $\text{I}$ but non-real factor $\text{III}$. We obtain $\det M_2 = r_0 \Sigma'_{z_2 v_2} \bar \Sigma'_{z_2 v_2}$ 
for some $r_0 \in \R$. Hence (\ref{clue}) gives:
$$(t^{u_2} - 1) (t^{-u_2} - 1) \det M_2 =$$ $$  r_0 (t^{z_2 v_2 u_2} - 1) (t^{-z_2 v_2 u_2} - 1) = r_0 (t^{v_2} - 1) (t^{-v_2} - 1)$$

\bigskip

\noindent$\bf{(e)}$ $h = k \geq 3$ and $j = 1$:  Here the determinant decomposition (\ref{III}) of 
$\det M_h$ has real factor $\text{III}$ but non-real factor $\text{I}$.  Compute that 
$\det M_h = r_0 \Sigma_{z_1 v_1} \bar \Sigma_{z_1 v_1}$ for some $r_0 \in \R$. Hence (\ref{clue}) gives
(recalling the formal convention $u_1 = 1$):
$$(\zeta_1^{u_1} t^{u_1} - 1) (\zeta_1^{-u_1} t^{-u_1} - 1) \det M_h = $$
$$r_0 (\zeta_1^{z_1 v_1 u_1} t^{z_1 v_1 u_1} - 1) (\zeta_1^{-z_1 v_1 u_1} t^{-z_1 v_1 u_1} - 1)=
r_0 (\zeta_1^{v_1} t^{v_1} - 1) (\zeta_1^{-v_1} t^{-v_1} - 1)$$

\bigskip

\noindent$\bf{(f)}$ $h = k \geq 3$ and $j = 2$: Here again, the determinant decomposition (\ref{III}) of 
$\det M_h$ has real factor $\text{III}$ but non-real factor $\text{I}$. Compute 
$\det M_h = r_0 \Sigma'_{z_2 v_2} \bar \Sigma'_{z_2 v_2}$ for some $r_0 \in \R$.
Recalling the formal convention $u_2 = 1$ we deduce from (\ref{clue}):
$$(t^{u_2} - 1) (t^{-u_2} - 1) \det M_h =$$ $$ r_0 (t^{z_2 v_2 u_2} - 1) (t^{-z_2 v_2 u_2} - 1) = 
r_0 (t^{v_2} - 1) (t^{-v_2} - 1)$$

\bigskip

\noindent$\bf{(g)}$ $h = k  = 2$ and $j = 1$:  In this case in the determinant decomposition (\ref{III}) of 
$\det M_2$  both factors $\text{I}$ and $\text{III}$ are non-real. Compute 
$\det M_2 = r_0 \Sigma_{z_1 v_1} \bar \Sigma_{z_1 v_1} \Sigma'_{z_2 v_2} \bar \Sigma'_{z_2 v_2}$ 
for some $r_0 \in \R$.  Hence (\ref{clue}) gives (recalling again $u_1 = 1$):
$$ (\zeta_1^{u_1} t^{u_1} - 1) (\zeta_1^{-u_1} t^{-u_1} - 1) (t^{u_2} - 1) (t^{-u_2} - 1) \det M_2 = $$
$$ r_0 (\zeta_1^{z_1 v_1 u_1} t^{z_1 v_1 u_1} - 1) (\zeta_1^{-z_1 v_1 u_1} t^{-z_1 v_1 u_1} - 1)
(t^{z_2 v_2 u_2} - 1) (t^{-z_2 v_2 u_2} - 1) = $$
$$r_0 (\zeta_1^{v_1} t^{v_1} - 1) (\zeta_1^{-v_1} t^{-v_1} - 1) (t^{v_2} - 1) (t^{-v_2} - 1)$$

\bigskip

\noindent$\bf{(h)}$ $h = k  = 1$ and $j = 2$:  Here too, in the determinant decomposition (\ref{III}) of $\det M_1$, 
both factors $\text{I}$ and $\text{III}$ are non-real. Compute 
$\det M_1 = r_0 \Sigma_{z_1 v_1} \bar \Sigma_{z_1 v_1} \Sigma'_{z_2 v_2} \bar \Sigma'_{z_2 v_2}$ 
for some $r_0 \in \R$.  Hence (\ref{clue}) gives (for $u_2 = 1$ as before):
$$(\zeta_1^{u_1} t^{u_1} - 1) (\zeta_1^{-u_1} t^{-u_1} - 1) (t^{u_2} - 1) (t^{-u_2} - 1) \det M_1 = $$
$$r_0 (\zeta_1^{z_1 v_1 u_1} t^{z_1 v_1 u_1} - 1) (\zeta_1^{-z_1 v_1 u_1} t^{-z_1 v_1 u_1} - 1)
(t^{z_2 v_2 u_2} - 1) (t^{-z_2 v_2 u_2} - 1) = $$
$$r_0 (\zeta_1^{v_1} t^{v_1} - 1) (\zeta_1^{-v_1} t^{-v_1} - 1) (t^{v_2} - 1) (t^{-v_2} - 1)$$

\bigskip

In order to finish the proof we have to ``paste together'' these calculations and verify, for each 
possibility of the values for the indices $j$ and $k$, that the equality (\ref{result-M}) is satisfied. 

To guide the reader through the various combinations,  the arguments needed in each case 
are assembled  into the following table: 

\smallskip
\begin{table}[!htbp]
\label{Table}

\centering
$\begin{array}{ c|c|c|c| }
 & k = 1 & k = 2 &  k \geq 3 \\
   & & & \\
\hline
   & & & \\
j = 1 & (b) &  (g) & (b) \, \& \, (e)\\
  & & & \\

\hline
  & & & \\
j = 2 & (h) & (a)  & (a) \, \& \, (f)  \\
 & & & \\
\hline
  & & & \\
j \geq 3  &  (b) \, \& \, (c)  & (a) \, \& \, (d) &  (a) \,\& \, (b) \\
 & & & \\
\hline
\end{array}$
\end{table}

\smallskip
\noindent Each of the nine cases is easily verified, where for the first two cases in the diagonal we also use the 
formal conventions $u_j = v_k = 1$. This completes the proof.
\end{proof}


\bigskip
\section{Generalizations}\label{sec:Generalization}

\bigskip
In this short final section we discuss how the results from the previous sections generalize to Fuchsian 
groups $G$  with associated quotient orbifold that is topologically a closed surface with handles or crosscaps.
In the orientable case the presentation given in (\ref{Fuchs}) becomes
\begin{equation}
\label{orientable-Fuchs}
\langle s_1, \ldots, s_\ell, a_1, b_1, \dots, a_g, b_g \mid s_1^{\gamma_1}, 
\ldots s_\ell^{\gamma_\ell}, \,\,s_1 \ldots s_\ell \overset{g}{\underset{ j = 1}{ \Pi}}  [a_j, b_j]\rangle
\end{equation}
with $\ell \geq 1, g \geq 1$ and all exponents $\gamma_k \geq 2$.

If the orbifold associated to $G$ is non-orientable, then there is at least one crosscap, and the
 corresponding presentation for $G$ is
\begin{equation}
\label{non-orientable-Fuchs}
\langle s_1, \ldots, s_\ell, c_1, \dots, c_h \mid s_1^{\gamma_1}, 
\ldots s_\ell^{\gamma_\ell}, \,\,s_1 \ldots s_\ell c_1^2 \ldots c_h^2
\rangle
\end{equation}
with  $\ell \geq 1, h \geq 2$ or $\ell \geq 2, h \geq 1$,
and all exponents $\gamma_k \geq 2$.

Consider, as before, {\em standard generating systems} $\cal U^*$ and $\cal V^*$ of $G$, which 
are obtained from $\cal U$ and $\cal V$ as in Theorem \ref{main} by 
$$\cal U^* = \cal U \cup \{a_1, b_1, \dots, a_g, b_g\} \quad \text{and} 
\quad \cal V^* = \cal V \cup \{a_1, b_1, \dots, a_g, b_g\}$$
in the orientable case, and by 
$$U^* = \cal U \cup \{c_1, \ldots, c_h\} \quad \text{and} \quad \cal V^* = \cal V \cup 
\{c_1, \ldots, c_h\} $$
in the non-orientable case. We then obtain:

\begin{cor}\label{main+}
Let $G$ be a group with presentation (\ref{orientable-Fuchs}) or  (\ref{non-orientable-Fuchs}), 
and let $\cal U^*$ and $\cal V^*$ be as defined above. In the orientable case assume that 
$m \geq 5$ if $n$ is even, and that $m \geq 7$ if $n$ is odd. In the non-orientable $n$ and $m$ must 
satisfy the same conditions, but with $n$
replaced by $n + h$.

Then $\cal U^*$ and $\cal V^*$ are Nielsen equivalent if and only if  $u_i = \pm v_i$ 
modulo $\gamma_i$, for all $i = 1, \ldots, \ell$.
\end{cor}

\begin{proof}
In the non-orientable case (\ref{non-orientable-Fuchs}) we quotient $G$ to a group with 
presentation as in (\ref{Fuchs}), by adding the relators $c_1^2, \ldots, c_h^2$. For the 
``only if'' direction we then use the observation (see Remark \ref{quotient-Nielsen})
that Nielsen equivalence is preserved when passing to a quotient group, while for the 
``if'' direction the same proof as given for Theorem \ref{main} applies.

In the orientable case we use the same proof as given in the previous sections: 
We extend the evaluation representations 
$$\eta: \Z G \to Sl_2(\Z[\langle t \mid t^p  \rangle])$$ 
from section \ref{sec:ProofOfMain} by mapping every $a_k$ and every $b_k$ to the 
unit matrix $I_2$. This extension method has  already been used in our previous paper \cite{LuMo:NEinFuch}, 
and all needed details are given there.
\end{proof}

\medskip

Alternatively to the quote given at the end of the last proof, one can also derive the argument directly 
from the material presented in the previous sections. This leads indeed to a much stronger statement, 
which we will sketch now:

\medskip

Consider any group $G$ with presentation
\begin{equation}
\label{non-orientable-Fuchs-gen}
\langle s_1, \ldots, s_\ell, d_1, \dots, d_q \mid s_1^{\gamma_1}, 
\ldots s_\ell^{\gamma_\ell}, \,\,s_1 \ldots s_\ell W
\rangle \, ,
\end{equation}
for an arbitrary element $W \in F(d_1, \ldots, d_q)$. Define generating systems
\begin{equation}
\label{gen-sys-+}
U^* = \cal U \cup \{d_1, \ldots, d_q\} \quad \text{and} \quad \cal V^* = \cal V \cup 
\{d_1, \ldots, d_q\} \, ,
\end{equation}
where $\cal U$ and $\cal V$ are standard generating systems of the 
quotient group 
$$G_0 = G/\langle\langle \{d_1, \ldots, d_q\}\rangle\rangle \, .$$ 
This group is clearly of type (\ref{Fuchs}) as considered in the previous sections. We note that if $G_0$ 
is non-exceptional, then the cyclic-faithful representation $\eta: G_0 \to Sl_2(\C)$ given by Lemma \ref{reps} 
(under the hypotheses stated there) lifts to a representation $\eta^*: G \to Sl_2(\C)$,  where every generator 
$d_k$ is mapped to the identity matrix $I_2$. 

As a consequence, all the arguments from the previous sections apply to $G$ as well, in particular the crucial 
argument in section 7 (proof of Lemma \ref{independent}): The $\eta^*$-image of the Fox derivatives 
$\partial R_0/\partial d_k$ vanish, independently of the choice of the element $W \in F(d_1, \ldots, d_q)$. 
This is because the formula (\ref{relation-R}) also holds for the generators $d_k$, so that 
$\partial R_0/\partial d_k$ contains 
the factor 
$$1+s_j+ s_j^2+ \ldots+s_j^{\gamma_j-1}$$ 
which is mapped by $\eta^*$ to 0. Hence only minor adaptations in the proof of Theorem \ref{main} are 
needed to give the following:

\begin{thm}
\label{W-extension}
Let $G$ be a group with presentation (\ref{non-orientable-Fuchs-gen}), and assume that the above 
quotient group $G_0$ is non-exceptional. 

Then the generating systems $\cal U^*$ and $\cal V^*$ as in (\ref{gen-sys-+}) are Nielsen equivalent if 
and only if $u_i = \pm v_i$ modulo $\gamma_i$, for all $i = 1, \ldots, \ell$.
\qed
\end{thm}

In fact, by restricting the choice of $W$ slightly, one can do even better, in that also many exceptional groups 
$G_0$ satisfy the conclusion of the above theorem. The details, however, will be provided elsewhere.

\end{document}